\documentclass[11pt,a4paper]{amsart}
\usepackage{hyperref}
\usepackage{mathrsfs}
\usepackage{amsmath}
\usepackage{amssymb}
\usepackage[all]{xy}
\usepackage{latexsym}
\usepackage[T1]{fontenc}
\usepackage{datetime}
\usepackage{pgf,tikz}
\usetikzlibrary{arrows}

 \usepackage[latin1]{inputenc}

\newcommand{ \inj}{ \hookrightarrow}
\newcommand{ \surj}{ \twoheadrightarrow}

\newcommand{\R}{\mathbb{R}}
\newcommand{\Z}{\mathbb{Z}}
\newcommand{\F}{\mathbb{F}}
\newcommand{\Q}{\mathbb{Q}}

\newcommand{\mf}{\underline{\F}}
\newcommand{\mz}{\underline{\Z}}
\newcommand{\ste}{\mathcal{A}}

\newcommand{\ab}{\mathcal{A}b}

\newcommand{\sh}{\mathcal{SH}}
\newcommand{\Sp}{\mathcal{S}p}
\newcommand{\T}{\mathcal{T}}
\renewcommand{\H}{\mathcal{H}}
\newcommand{\M}{\mathcal{M}}
\newcommand{\Or}{\mathcal{O}}

\newcommand{\freem}{\mathrm{Fr}\M^f}
\newcommand{\torsm}{\mathrm{Tors}\M^f}
\newcommand{\free}{\mathrm{Fr}}
\newcommand{\tors}{\mathrm{Tors}}

\theoremstyle{definition}
\newtheorem{de}{Definition}[section]
\newtheorem{nota}[de]{Notation}

\newtheorem*{de*}{Definition}

\theoremstyle{plain}
\newtheorem{thm}[de]{Theorem}
\newtheorem{lemma}[de]{Lemma}
\newtheorem{pro}[de]{Proposition}
\newtheorem{cor}[de]{Corollary}

\newtheorem*{thm*}{Theorem}
\newtheorem*{lemma*}{Lemma}
\newtheorem*{pro*}{Proposition}
\newtheorem*{cor*}{Corollary}

\theoremstyle{remark}
\newtheorem{rk}[de]{Remark}
\newtheorem{ex}[de]{Example}

\title[Anderson duality]{Equivariant Anderson duality and Mackey functor duality}
\author{Nicolas Ricka}
\address{Laboratoire Analyse, G\'eom\'etrie et Applications \\
 UMR 7539 \\
 99, avenue Jean-Baptiste Cl\'ement, \\
 93430 Villetaneuse, France}
\email{ricka@math.univ-paris13.fr}

\begin{document}

\begin{abstract}
We show that the $\Z/2$-equivariant $n^{th}$ integral Morava $K$-theory with reality is self-dual with respect to equivariant Anderson duality. In particular, there is a universal coefficients exact sequence in integral Morava K-theory with reality, and we recover the self-duality of the spectrum $KO$ as a corollary.
The study of $\Z/2$-equivariant Anderson duality made in this paper gives a nice interpretation of some symmetries of $RO(\Z/2)$-graded ({\it i.e.} bigraded) equivariant cohomology groups in terms of Mackey functor duality.
\end{abstract}

\maketitle

\textbf{Conventions:} In this paper, $\F$ denotes the field with two elements. When considering the Steenrod algebra and the chromatic tower, the prime number is assumed to be $p=2$. The category of abelian groups is denoted $\ab$. For $E$ an object in the category of spectra (resp. $\Z/2$-equivariant spectra), we denote $E^*$ (resp. $E^{\star}$) the cohomology theory represented by $E$, and $E_*$ (resp. $E_{\star}$) the homology theory represented by $E$. The homotopy of $E$ is denoted $E_*$ (resp. $E_{\star}$). Equivariant cohomology theories are graded over the orthogonal representation ring, thus $\star$ is an orthogonal representation of $\Z/2$. We denote $1$ the trivial one dimensional representation, and $\alpha$ the sign representation.

\section{Introduction}
The focus of this paper is to generalise the universal coefficient exact sequence to $\Z/2$-equivariant generalised cohomology theories.
In \cite{And69}, Anderson introduced a duality functor $\nabla^e$ of the category of non-equivariant spectra, and showed that for a spectrum $E$, the $E$-cohomology and $E$-homology enters in a short exact sequence 
$$ 0 \rightarrow Ext^1_{\Z}(E_{*-1}(X),Z) \rightarrow (\nabla^e E)^{*}(X) \rightarrow Hom_{\Z}(E_{*}(X),\Z) \rightarrow 0$$
which restricts to the usual universal coefficients exact sequence for ordinary (co)homology.

Let $KO$ be the non-equivariant spectrum which represent periodic real K-theory, and $KU$ the non-equivariant spectrum which represent periodic complex K-theory. The aim of \cite{And69} is to provide universal coefficients in K-theories (both complex and real) by showing that the periodic K-theory spectrum is self-dual for the duality functor $\nabla^e$. As an application, Anderson obtains universal coefficient exact sequences computing K-cohomology from K-homology. 

Recently, in \cite{HS14, Sto11}, the authors show that the spectra $KO$ and a form of topological modular forms are self-Anderson-dual, up to suspension.
Recall that the complex conjugation on complex vector bundles induces a $\Z/2$-action on the spectrum $KU$, and that this produces a genuine $\Z/2$-equivariant spectrum $K\R$. This $\Z/2$-spectrum was introduced by Atiyah in \cite{At66}, and the associated (co)homology theory is called K-theory with reality. The fixed points of $KU$ with respect to its $\Z/2$-action is the spectrum $KU^{\Z/2} = KO$.
A byproduct of \cite{HS14} is that the weak equivalence $\nabla^{e}KU \cong KU$ realising the self-Anderson duality of $KU$ is compatible with the $\Z/2$-action induced by complex conjugation.

Let $K(n)$ be the $n^{th}$ integral Morava K-theory spectrum. These are higher chromatic version of $KU$, and there is an weak equivalence $KU \cong K(1)$. The homotopy groups of $K(n)$ are $\Z[v_n^{\pm 1}]$, where $v_n$ lies in degree $2(2^n-1)$. There are also $\Z/2$-equivariant refinements of these spectra, denoted $K\R(n)$, whose definition appears in \cite{HK01} and is recalled here in Definition \ref{de:moravareal}. It should be noted that one major difference between the spectra $K\R(n)$ for $n=1$ and $n>1$ is that only $K\R(1)$ is known to be a ring 
spectrum. For simplicity, $v_n$ also denotes a well-chosen $\Z/2$-equivariant lift of the classes $v_n \in K(n)_{2(2^n-1)}$.

The aim of this paper is to generalize the results described above in two directions:
\begin{itemize}
 \item show that the self duality of $KU$ is not only compatible with the action of $\Z/2$, but is really a self-duality result for the genuine $\Z/2$-equivariant spectrum $K\R$, with respect to a $\Z/2$-equivariant version of Anderson duality which is defined below,
 \item prove that this is a particular occurrence of a more general result of self-Anderson duality of the $\Z/2$-spectra $K\R(n)$.
\end{itemize}
Our approach differs from \cite{HS14} in the use of the slice spectral sequence of Hill-Hopkins-Ravenel \cite{HHR} instead of the homotopy fixed points to compute the homotopy groups of the spectra $K\R(n)$. The fact that slices of $K\R(n)$ are the $\Z/2$-spectrum $H\mz$ yields the result, without computing any differential.

Let $MU$ be the complex cobordism spectrum. The action of $\Z/2$ on the spaces appearing in the definition of $MU$ assemble to define a commutative ring $\Z/2$-spectrum $M\R$. As $M\R$ is a commutative ring spectrum, it makes sense to speak about the category of $M\R$-modules, as well as the category of modules over various commutative ring spectra build from $M\R$. For example, let $M\R[v_n^{\pm 1}]$ be the $M\R$-module localisation of $M\R$ with respect to $v_n$. Explicitly, $M\R[v_n^{\pm 1}] = \underset{k}{\mathrm{hocolim}} \Sigma^{-k|v_n|}M\R$.
The main result of this paper is the following:

\begin{thm*}[\emph{Theorem \ref{thm:selfand}}]
The map $M\R[v_n^{-1}] \rightarrow \Sigma^{-2+2\alpha} \nabla K\R(n)$ induced by a generator $1 \in \pi_0(\Sigma^{-2+2\alpha}K\R(n))$ factorizes through $K\R(n)$, and yields a weak equivalence of $M\R$-modules
$$ K\R(n) \stackrel{\cong}{\rightarrow} \Sigma^{-2+2\alpha} \nabla K\R(n).$$
\end{thm*}

\begin{de*}
 Let $KO(n)$ be the spectrum $K\R(n)^{\Z/2}$.
\end{de*}

As an application of Theorem \ref{thm:selfand}, we show the following result, passing to fixed points:

\begin{cor*}[\emph{Corollary \ref{cor:koselfand}}]
There is a weak equivalence of non-equivariant spectra
$$ KO(n) \cong \Sigma^{-2^{n+2}+4} \nabla^e KO(n).$$
\end{cor*}

\textbf{Organisation of the paper:}
The key ingredient of this result is the study of equivariant Anderson duality, which is the subject of section \ref{sec:anderson}. The corollary then follows from the strong completion property for $K\R(n)$ given in Lemma \ref{lemma:completion}, and the good relationship between Anderson duality and fixed points for complete $\Z/2$-spectra (Proposition \ref{pro:fixeddual}). The definition of the equivariant Anderson duality functor is straightforward from the non-equivariant definition (see subsection \ref{sub:anddual}), but its behaviour with Eilenberg-MacLane spectra is richer than in the non-equivariant case: it acts non-trivially on equivariant Eilenberg-MacLane spectra, even those with finitely generated torsion free coefficients. We show in Proposition \ref{pro:relation:dualities} that it is a topological version of the Mackey functor duality of \cite[chapter 4]{TW95}.
This gives a nice interpretation of the symmetry appearing in various cohomology rings. For example, if $H\mf$ denotes the Eilenberg-MacLane $\Z/2$-spectrum with coefficients in the constant Mackey functor $\F$, the symmetries of the bigraded Mackey functor $H\mf^{\star}$, appearing {\it e.g.} \cite{FL04}, fits nicely in this context.

These results allow us to use duality arguments on the $E_2$-page of the slice spectral sequence for $K\R(n)$, and build a morphism of $\Z/2$-spectra $$ K\R(n) \stackrel{\cong}{\rightarrow} \Sigma^{-2+2\alpha} \nabla K\R(n)$$ that realizes an isomorphism on the $E_2$-page of the spectral sequence.

The efficiency of our approach resides in the fact that the comparison between the two spectral sequences do not involve any more computation than the equivariant homology of a point.

\tableofcontents

\section{Elements of $\Z/2$-equivariant stable homotopy theory}

\subsection{Equivariant stable homotopy theory and Mackey functors} \label{sub:conventions}

We refer to \cite{MM} for the constructions and definitions in the equivariant stable homotopy category. 

\begin{nota}
Let $\Z/2\T$ be the category of $\Z/2$-spaces and $\Z/2\H$ its homotopy category with respect to the usual fine model structure. It is called {\it~the $\Z/2$-equivariant homotopy category}.
\end{nota}

We now define the spheres and suspension functors we need to set up the equivariant stable homotopy category.

\begin{de} \label{de:repsphere}
A {\em representation sphere} is a pointed $\Z/2$-space of the form $S^V$ (the one point compactification of $V$), for $V$ an orthogonal representation of $\Z/2$.
\end{de}

\begin{de}
The {\it~suspension functors} are the functors of the form $$S^V \wedge (-) : \Z/2\T \rightarrow \Z/2\T,$$
for $V$ an orthogonal representation of $\Z/2$. These functors are denoted $\Sigma^V$.
\end{de}

Forcing the suspension functors $\Sigma^V$ to be invertible up to weak equivalence, for all orthogonal representation $V$, yields the category of $\Z/2$-equivariant spectra (precisely, the category of $\Z/2$-equivariant orthogonal spectra indexed over a complete universe).

\begin{nota}
 Let $\Z/2\Sp$ be the category of $\Z/2$-equivariant spectra, and $\Z/2\sh$ the {\it~$\Z/2$-equivariant stable homotopy category, i.e.} its homotopy category. Let $\Sigma^{\infty} : \Z/2\T \rightarrow \Z/2\sh$ be the functor defined by the universal property of $\Z/2\Sp$.
 \end{nota}

Since the functors $S^V \wedge (-)$ are now weakly invertible in the equivariant stable homotopy category, there is a $\Z/2$-spectrum, $S^{-V}$, characterised by the existence of a weak equivalence $S^V \wedge S^{-V} \cong S^0$. This assembles to a nice functor $S^{(-)}$, from the group completion of the monoid of orthogonal representations to $\Z/2\sh$.

\begin{de}
Let $RO(\Z/2)$ be the {\it~real representation ring of $\Z/2$}, that is the group completion of the monoid of orthogonal representation of $\Z/2$, with the direct sum of representations as the addition, and the tensor product as multiplication.
\end{de}

Recall that $RO(\Z/2)$ is a free abelian group on two generators $\Z\{1,\alpha\}$, where $1$ stands for the one dimensional trivial representation, and $\alpha$ for the one dimensional sign representation.

The two $\Z/2$-spectra $\Z/2_+ = \Sigma^{\infty} \Z/2_+$ and $\frac{\Z/2}{\Z/2}_+ = \Sigma^{\infty}S^0$ plays a particular role in this category.

\begin{de}
Let $\Or$ be the full subcategory of $\Z/2\sh$ whose objects are $\Z/2_+$ and $\frac{\Z/2}{\Z/2}_+$.
\end{de}

\begin{de}
A {\it~Mackey functor} is an additive functor $\Or^{op} \rightarrow \ab$. Let $\M$ be the category of Mackey functors and natural transformations between them.
Let $\M^{RO(\Z/2)}$ be the category of $RO(\Z/2)$-graded Mackey functors.
\end{de}

The two stable maps $\Z/2_+ \rightarrow S^0$, which collapses $\Z/2$ onto the non base point of $S^0$, and the transfer map $S^0 \rightarrow \Z/2_+$ (which exists only in the stable category) induces two morphisms of abelian groups, for all Mackey functor $M$:
$$ \rho : M(S^0) \rightarrow M(\Z/2_+)$$
and 
$$\tau : M(\Z/2_+) \rightarrow M(S^0)$$
called the restriction and the transfer of $M$, respectively (see \cite[IX,4-5]{May96}).

Let $X,Y$ be $\Z/2$-spectra. The stable homotopy classes of morphisms $[X,Y]_{\star}$ is naturally a $RO(\Z/2)$-graded Mackey functor via the following construction: for all $o \in \Or$,
$$[X,Y]_V(o) := [o \wedge \Sigma^V X, Y]^{\Z/2},$$
where $[-,-]^{\Z/2}$ is the abelian group of equivariant stable homotopy classes of maps.

With the previous notations, $[-,-]_{\star}$ defines a functor
$$\Z/2\sh^{op} \times \Z/2\sh \rightarrow \mathcal{M}^{RO(\Z/2)}.$$

\begin{de}
Let $E$ be a $\Z/2$-spectrum. The {\it~homotopy Mackey functor} of $E$ is by definition the $RO(\Z/2)$-graded Mackey functor $\underline{\pi}_{\star}(E) := [S^0,E]_{\star}$.
\end{de}

\begin{nota}
To ease the notations, for a Mackey functor $M \in \M$, we denote $M_{e}$, or $M^e$ if the subscript is already taken, the abelian group $M(\Z/2_+)$. Similarly, $M_{\Z/2}$, or $M^{\Z/2}$ if the subscript is already taken, denotes the abelian group $M(S^0)$.
\end{nota}

The interested reader can consult \cite{FL04} for Mackey functors and the relationship between Mackey functors and equivariant stable homotopy theory. In particular, we have the following graphical representation of Mackey functors:

\begin{nota} \label{nota:mackey}
Let $M$ be a Mackey functor. Let $\theta_M :  M_e \rightarrow M_e$ the action of the non-trivial element of $\Z/2$ on $M_e$, we represent $M$ by the following diagram:
$$\xymatrix{ M_{\Z/2}  \ar@/^/[d]^{ \rho_M } \\ 
M_e. \ar@/^/[u]^{\tau_M} } $$
Observe that, since $\theta_M = \rho_M\tau_M -1$ this diagram suffices to determine $M$, and in particular the action of $\Z/2$ on $M_e$.
\end{nota}

We end this section with some properties of the category of Mackey functors.

The following definition of the monoidal structure on Mackey functors for the group $\Z/2$ is taken from \cite[p.11]{FL04}.

\begin{de} \label{de:convolution}
Let $M,N \in \mathcal{M}$. Define $M \boxtimes N$ to be the diagram of abelian groups: 
$$ \xymatrix{ (M_{\Z/2} \otimes N_{\Z/2}) \oplus (M_e \otimes N_e)/  \simeq \ar@/^/[d]^{ ( \rho_M \otimes \rho_N, tr)} \\ 
M_e \otimes N_e, \ar@/^/[u]^{ i_2} \ar@(dl,dr)[]_{ \theta_M \otimes \theta_N} } $$
where the map $tr$ is the trace map for the $\Z/2$-action, and $i_2$ is induced by the inclusion. 
The equivalence relation $ \simeq$ is generated by $$m_e \otimes \tau_N n_{\Z/2} \simeq \rho_M m_e \otimes n_{\Z/2}$$ and $$\tau_M m_{\Z/2} \otimes n_e \simeq m_{\Z/2} \otimes \rho_N n_e$$ with evident notations.
This is part of a closed symmetric monoidal structure on the category of Mackey functors, and we denote $\underline{Hom}_{\M}(-,-)$ the internal hom functor.
\end{de}

\begin{pro}[\emph{\cite{FL04}}] \label{pro:monoidal}
The product $ \boxtimes$ gives $ \mathcal{M}$ the structure of a closed symmetric monoidal category. Let $A$ be the unit for this structure. The Mackey functor $A$ is called the Burnside ring Mackey functor, and is determined by the diagram

$$\xymatrix{ \Z \oplus \Z  \ar@/^/[d]^{ (1,2) } \\ 
\Z. \ar@/^/[u]^{(0,1)} } $$

In particular, there are isomorphisms $id_{\M} \boxtimes A \cong id_{\M}\cong A \boxtimes id_{\M}$ and $\underline{Hom}_{\M}(A,-) \cong id_{\M}$.
\end{pro}

As functors of the form $[-,E]^{\star}$, for $E \in \Z/2\Sp$, $\Z/2$-equivariant cohomology theories are functors $\Z/2\sh^{op} \rightarrow \mathcal{M}^{RO(\Z/2)}$. They are precisely the functors satisfying equivariant analogues of the Eilenberg-Steenrod axioms. 

\begin{nota}
To make the notation more readable, we will denote $E^{\star}_e(X)$ and $E^{\star}_{\Z/2}(X)$ for  $E^{\star}(X)_e$ and $E^{\star}(X)_{\Z/2}$ respectively.
\end{nota}

\subsection{Postnikov towers and ordinary cohomology theories}

Equivariant Postnikov towers provide the appropriate notion of ordinary cohomology theory.

\begin{pro} \label{pro:functorH}
The $\Z/2$-equivariant Postnikov tower defines a $t$-structure on the $\Z/2$-equivariant stable homotopy category whose heart is isomorphic to the category $\mathcal{M}$ of Mackey functors for the group $\Z/2$. In particular, one has an Eilenberg-MacLane functor
$$H : \mathcal{M} \rightarrow \Z/2\sh$$
which sends a short exact sequences of Mackey functors to a distinguished triangle of $\Z/2$-equivariant spectra.
\end{pro}

\begin{proof}
This proposition summarize the results of \cite[Proposition I.7.14]{LMS} and \cite[Theorem 1.13]{Le95} in the particular case of the group with two elements.
\end{proof}

We now define some classical Mackey functor which play a role in this paper.

\begin{de} \label{de:partmack}
Let $C$ be an abelian group and $N$ be a $\Z[\Z/2]$-module. Denote $\frac{N}{\Z/2}$ the quotient of $N$ by the action of $\Z/2$. \\

\vspace*{0,2cm}
\begin{tabular}{|c|c||c|c|}
\hline
Notation & Mackey functor & Notation & Mackey functor \\

\hline \hline
 $\underline{C}$ &  $\xymatrix{ C \ar@/^/[d]^{ =} \\ 
C \ar@/^/[u]^{ 2} } $ & $\underline{C}^{op}$ &  $\xymatrix{ C \ar@/^/[d]^{ 2} \\ 
C \ar@/^/[u]^{=} } $\\
\hline
$[C]$ & $\xymatrix{ 0\ar@/^/[d]^{0} \\ 
C \ar@/^/[u]^{0} } $ &  $<C>$ &  $\xymatrix{ C \ar@/^/[d]^{0} \\ 
0 \ar@/^/[u]^{0} } $ \\

\hline
$R(N)$ & $ \xymatrix{ N^{\Z/2} \ar@/^/[d] \\ 
N \ar@/^/[u]^{ trace } } $ & $ L(N)$ &  $\xymatrix{ \frac{N}{\Z/2} \ar@/^/[d] \\ 
N \ar@/^/[u]^{\pi} } $\\
\hline

\end{tabular}
\vspace*{1cm}
\end{de}

\begin{rk}
\begin{itemize}
\item The notation $(-)^{op}$ is supposed to indicate that the Mackey functor $M^{op}$ is obtained from $M$ by exchanging the restriction and transfer morphisms. This is part of a more general construction which we will not make explicit here.
\item One must be careful with these diagrams, the fact that the composite $\rho \tau = trace$ implies that the action of $\Z/2$ on $[C]_{\Z/2} = C$ is by multiplying by $-1$. This construction only makes sense when the ambient group is $\Z/2$.
\item The Mackey functors whose restrictions are monomorphisms, and $\mf$, $\mz$ in particular, play a special role in this context. One reason is that equivariant Eilenberg-MacLane spectra with coefficients in these Mackey functors are exactly the $0th$-slices of Hill-Hopkins-Ravenel's slice filtration (see \cite[Proposition 4.50 (ii)]{HHR}), and that $H\mz$ is the $0th$ slice of the sphere spectrum by \cite[Corollary 4.54]{HHR}.
\end{itemize}
\end{rk}

\subsection{Some properties of Mackey functors arising as homotopy groups}

We now turn to a general result about Mackey functors obtained from cohomology theories. We first give additional structure on the $RO(\Z/2)$-graded abelian group valued functor $[-,-]^{\Z/2}_{\star}$.

\begin{de}
Let $a \in \underline{\pi}_{-\alpha}(S^0)$ be the class represented by the inclusion of fixed points $S^0 \rightarrow S^{\alpha}$. The class $a$ is called the Euler class of $\alpha$.
\end{de}

For any $\Z/2$-spectra $X$ and $Y$, the $RO(\Z/2)$-graded abelian group $[X,Y]^{\Z/2}_{\star}$ has a natural action of $\underline{\pi}_{\star}^{\Z/2}(S^0)$. Consider the natural $\Z[a]$-module structure on the  $[X,Y]^{\Z/2}_{\star}$ restricted from this action.

\begin{lemma} \label{lemma_mackey_a}
Let $E$ be a $ \Z/2$-spectrum such that the action of $2a$ on $E_{\star}^{\Z/2}$ is trivial. The $RO(\Z/2)$-graded Mackey functor $E_{\star}$ satisfies the following properties:
\begin{enumerate}
 \item $Im(a)= Ker( \rho)$.
\item $ Ker(a) = Im( \tau)$.
\item Let $x$ be an element of $E_{ \star}^{\Z/2}$. Suppose that $x$ is divisible by $a$, and that $x$ is not in $ Ker(a)$. Then $x$ induces a split monomorphism of Mackey functors $< \F> \inj \underline{ E_{ \star}}$.
\item Suppose that $ E^{e}_{\star}$ has no $2$-torsion. Then an element  $x \in E_{\star}^{\Z/2}$ is divisible by $a$ if and only if $2x=0$.
\end{enumerate}
\end{lemma}

\begin{proof}
 The two first points are shown using similar methods. Recall that there is a cofiber sequence
 $$ \Z/2_+ \rightarrow S^0 \rightarrow S^{ \alpha}.$$

\begin{enumerate}
 \item Apply the exact functor $ [ -, \Sigma^{ - \star} E]^{\Z/2}$ to the previous cofiber sequence. Then, we have isomorphisms:
$$ \xymatrix{  [S^{ \alpha}, \Sigma^{- \star} E]^{\Z/2} \ar[r] \ar@{ = }[d] & [S^0, \Sigma^{- \star} E]^{\Z/2} \ar[r] \ar@{ = }[d] & [ \Z/2_+, \Sigma^{- \star} E]^{\Z/2} \ar@{ = }[d] \\
\underline{ \pi}_{ \star+ \alpha}^{\Z/2}(E) \ar[r]^a & \underline{ \pi}_{ \star}^{\Z/2}(E) \ar[r]^{ \rho} & \underline{ \pi}_{ \star}^{e}(E) \\ }$$
where the rows are exact.
\item Apply the exact functor $ [ S^{ \star}, (-) \wedge E]^{\Z/2}$ to the cofiber sequence. Then, we have isomorphisms:
$$ \xymatrix{  [S^{ \star}, \Z/2_+ \wedge E]^{\Z/2} \ar[r] \ar@{ = }[d] & [S^{ \star}, E]^{\Z/2} \ar[r] \ar@{ = }[d] & [ \Sigma^{ \star - \alpha}, E]^{\Z/2} \ar@{ = }[d] \\
\underline{ \pi}_{ \star}^{e}(E) \ar[r]^{ \tau} & \underline{ \pi}_{ \star}^{\Z/2}(E) \ar[r]^{a} & \underline{ \pi}_{ \star}^{\Z/2}(E) \\ }$$
where the rows are exact.
\item We use the two first points for the element $x$. By the second point, $x \not\in Im( \tau)$, and by the second one, $x \in Ker( \rho)$. Now, the fact that $x$ is divisible by $a$ implies that $2x = 0$ because $2a=0$. Thus, the Mackey functor monomorphism
 $$ \xymatrix{\F \ar@/^/[d] \ar@{^(->}[rr] & &  E_{ \star}^{\Z/2} \ar@/^/[d] \\ 
0 \ar@/^/[u] \ar[rr] &  & E_{ \star}^e, \ar@/^/[u] } $$ induced by $x$ is split.
\item The last point is a consequence of the fact that $2a=0$ and the first point.
\end{enumerate}
\end{proof}

\begin{rk}
Even though the hypothesis of the last proposition could seem restrictive, the class of $\Z/2$-spectra such that $2a$ acts as $0$ contains very fundamental examples, the motivating class for our purpose are $M\R$-modules, since $2a = 0$ in $M\R_{-\alpha}$ (this is a consequence of the computations of \cite{HHR} by the slice spectral sequence).
\end{rk}

\subsection{The isotropy separation cofibre sequence}

We conclude this section with the {\em isotropy separation} cofibration sequence.

\begin{de}
Let $E\Z/2$ be the universal $\Z/2$-space. Let $\widetilde{E\Z/2}$ be the cofiber of the map $E\Z/2_+ \rightarrow S^0$ collapsing $E\Z/2$ to the non base point of $S^0$.
\end{de}

\begin{pro}
There is a $\Z/2$-homotopy equivalence $\widetilde{E\Z/2} = \mathrm{colim}(S^{k \alpha})$, where the colimit is taken over the inclusions $S^{k \alpha} \rightarrow S^{(k+1)\alpha}$.
\end{pro}

\begin{de}
The isotropy separation cofibration sequence is the cofibration of pointed $\Z/2$-spaces
$$ E\Z/2_+ \rightarrow S^0 \rightarrow \widetilde{E\Z/2}.$$
\end{de}

In particular, for all $\Z/2$-spectrum $E$, there is a diagram, whose lines and columns are cofibre sequences

$$ \xymatrix{ E\Z/2_+ \wedge F(\widetilde{E\Z/2}, E) \ar[r] \ar[d] & F(\widetilde{E\Z/2}, E) \ar[d] \ar[r]^{\cong} & \widetilde{E\Z/2} \wedge F(\widetilde{E\Z/2}, E) \ar[d]  \\
E\Z/2_+ \wedge E \ar[d]^{\cong} \ar[r] & E \ar[d] \ar[r] & \widetilde{E\Z/2} \wedge E \ar[d]\\
E\Z/2_+ \wedge F(E\Z/2_+, E) \ar[r] & F(E\Z/2_+, E) \ar[r] & \widetilde{E\Z/2} \wedge F(E\Z/2_+, E),\\ }$$

where $F(-,-)$ is the function $\Z/2$-spectrum.

\begin{de}
A $\Z/2$-spectrum $E$ is called {\it~complete} if the map $$E\Z/2_+ \wedge E \rightarrow E$$ is a weak equivalence, or equivalently if $\widetilde{E\Z/2} \wedge E \cong 0$.
\end{de}

\section{Equivariant Anderson duality and Mackey functor duality} \label{sec:anderson}

\subsection{Mackey functor duality} We now turn to the appropriate notion of duality in the category of Mackey functors. It coincides with the duality studied in \cite[chapter 4]{TW95} when working over a field.

Recall from \cite{LMS} that the category of orbits $ \mathcal{O}$ (see subsection \ref{sub:conventions}) is self-dual, precisely there is an isomorphism of additive categories $ \phi : \mathcal{O}^{op} \rightarrow \mathcal{O}$.

\begin{lemma}
Let $F: \ab^{op} \rightarrow \ab$ an additive functor. Then the assignment $M \in \M \mapsto F \circ M^{op} \circ \phi$ defines a functor $\M F : \M^{op} \rightarrow \M$.
\end{lemma}

\begin{proof}
The composite is obviously an additive functor $\mathcal{O}^{op} \rightarrow \ab$.
\end{proof}

This lemma is used to define various endofunctors of the category of Mackey functors.
\begin{de}
\begin{enumerate}
\item Let
$$\nabla : \M^{op} \rightarrow \M$$ be the functor $\M(-)^{\vee}$, where $(-)^{\vee} = Hom_{\Z}(-,\Z)$. Explicitly, it sends a Mackey functor
$$\xymatrix{ M_{\Z/2}  \ar@/^/[d]^{ \rho_M } \\ 
M_e \ar@/^/[u]^{\tau_M} } $$
to
$$\xymatrix{ M_{\Z/2}^{\vee}  \ar@/^/[d]^{ \tau_M^{\vee} } \\ 
M_e^{\vee}, \ar@/^/[u]^{\rho_M^{\vee}} } $$
and acts similarly on morphisms.
\item Let
$$\nabla_{tors} : \M^{op} \rightarrow \M$$ be the functor $\M Ext^1(-,\Z)$. Explicitly, it sends a Mackey functor
$$\xymatrix{ M_{\Z/2}  \ar@/^/[d]^{ \rho_M } \\ 
M_e \ar@/^/[u]^{\tau_M} } $$
to
$$\xymatrix{ Ext^1(M_{\Z/2},\Z)  \ar@/^/[d]^{ Ext^1(\tau_M, \Z) } \\ 
Ext^1(M_e, \Z), \ar@/^/[u]^{Ext^1(\rho_M,\Z)}. } $$
\end{enumerate}
\end{de}

The functors $(-)^{\vee}$ and $Ext^1_{\Z}(-,\Z)$ behaves particularly well with respect to finitely generated objects.

\begin{de}
Let $\M^f$ be the full subcategory of $\M$ whose objects are Mackey functors which are level-wise finitely generated abelian groups.
\end{de}

\begin{de}
Let $\freem$ (resp. $\torsm$) be the full subcategory of $\M^f$ whose objects are Mackey functors $M$ such that $M_e$ and $M_{\Z/2}$ are free $\Z$-modules (resp. are torsion $\Z$-modules).
\end{de}

\begin{de}
Let $\free : \M^f \rightarrow \freem$ be the functor which sends a Mackey functor $M$ to the image of $M \rightarrow M \otimes_{\Z} \mathbb{Q}$ and $\tors : \M^f \rightarrow \torsm$ the functor which sends $M$ to the kernel of $M \rightarrow M \otimes_{\Z} \mathbb{Q}$.
\end{de} 

\begin{lemma}
There is a unique functorial short exact sequence
$$ 0 \rightarrow \tors \rightarrow Id_{\M^f} \rightarrow \free \rightarrow 0$$
of functors $\M^f \rightarrow \M^f$.
\end{lemma}

\begin{proof}
Unicity comes from the unicity of such an exact sequence at the level of abelian groups, existence comes from the construction.
\end{proof}

\subsection{Equivariant Anderson duality} \label{sub:anddual} We now discuss a $\Z/2$-equivariant version of Anderson duality, and its relation with Mackey functors. Anderson duality was introduced by Anderson in \cite{And69}, and used in \cite{Kai71} and more recently in \cite{Sto11,HS14}. The goal of this section is to understand equivariant Anderson duality in terms of Mackey functors.

We start by a definition of Anderson duality, following \cite{Kai71, HS14}. The first step is to define Brown-Comenetz duality, as in \cite{BC76}

\begin{pro} \label{pro:bcdual}
Let $I$ be an injective abelian group. Then the assignment 
$$(d_I)^{\star}_{\Z/2} : X \mapsto Hom_{\Z}(\underline{\pi}_{-\star}^{\Z/2}(X),I)$$
defines a cohomology theory.
\end{pro}

\begin{proof}
This is like the classical case. The appropriate Eilenberg-Steenrod axioms are provided by \cite[XIII.1 and XIII.2]{May96}, where exactness is satisfied because $I$ is injective. 
\end{proof}

Recall from \cite[XIII.1 and XIII.2]{May96} that $\Z/2$-equivariant cohomology theories are the functors represented by $\Z/2$-spectra.

\begin{de}
Let $d_I$ be the $\Z/2$-spectrum representing $(d_I)^{\star}_{\Z/2}$.
\end{de}

Now, by Brown's representability theorem, a map between injective abelian groups $I \rightarrow J$ induces a morphism of $\Z/2$-spectra $d_I \rightarrow d_J$. We apply it to the injective resolution $$\Q \surj \Q/\Z$$ of $\Z$.

\begin{de} \label{de:anddual}
Let $d_{\Z} \in \Z/2\sh$ be the homotopy fiber of the map $d_{\mathbb{Q}} \rightarrow d_{\mathbb{Q}/\Z}$.
Let $\nabla$ be 
$$ F(-,d_{\Z}) : \Z/2\sh^{op} \rightarrow \Z/2\sh.$$
The functor $\nabla$ is called the {\it~equivariant Anderson duality functor}.
\end{de}

\begin{rk} \label{rk:nablaexact}
The functor $ \nabla$ sends cofibre sequence of spectra to cofibre sequences of spectra, because of the analogous property for the function spectrum $F(-,-)$ in both variables.
\end{rk}

Since its discovery, the aim of Anderson duality is to produce universal coefficient theorems. We now state the universal coefficient exact sequence in the $\Z/2$-equivariant setting.

\begin{pro} \label{pro:univcoeff}
Let $E$ and $X$ be $\Z/2$-spectra. Then there is a short exact sequence of Mackey functors, called the equivariant universal coefficients exact sequence, natural in $X$ and $E$:
$$ 0 \rightarrow \nabla_{tors}(E_{\star-1}(X)) \rightarrow (\nabla E)^{\star}(X) \rightarrow \nabla(E_{\star}(X)) \rightarrow 0.$$
\end{pro}

\begin{proof}
We start by showing that the the sequences of abelian groups obtained after applying the evaluation functors $(-)_e : \M \rightarrow \ab$ and $(-)_{\Z/2}: \M \rightarrow \ab$ are exact.

Let $Y = E \wedge X$. With this definition, $E^{\Z/2}_{\star}(X) = \underline{\pi}^{\Z/2}_{\star}(Y)$. We have 
\begin{eqnarray*}
\nabla Y & = & F(E \wedge X, d_{\Z}) \\
& \cong & F(X, F( E, d_{\Z})) \\
& \cong & F(X, \nabla E),
\end{eqnarray*}
so that $(\nabla E)_{\Z/2}^{\star}(X) = \underline{\pi}^{\Z/2}_{-\star}(\nabla Y)$.

Now, by definition of the cohomology theories $(d_I)^{\star}_{\Z/2}$ for the abelian groups $\mathbb{Q}$ and $\mathbb{Q}/\Z$, the long exact sequence associated to the cofiber sequence
$$ \nabla Y \rightarrow F(Y,d_{\mathbb{Q}}) \rightarrow F(Y,d_{\mathbb{Q}/\Z})$$
obtained by applying the functor $\underline{\pi}^{\Z/2}_{-\star}(-)$ is
$$\hdots \rightarrow Hom_{\Z}( \underline{\pi}^{\Z/2}_{\star-1}(Y), \mathbb{Q}/\Z) \rightarrow \underline{\pi}^{\Z/2}_{-\star}(\nabla Y) \rightarrow Hom_{\Z}(\underline{\pi}^{\Z/2}_{\star}(Y), \mathbb{Q}) \rightarrow \hdots.$$
Now, the first and last arrows appear in the complex computing $\mathrm{Ext}^*_{\Z}(\underline{\pi}^{\Z/2}_{\star}(Y), \Z)$ obtained by taking the injective resolution of abelian groups $\mathbb{Q} \rightarrow \mathbb{Q}/\Z$ of $\Z$, so that the kernel of the last arrow is $$\mathrm{Hom}_{\Z}(\underline{\pi}^{\Z/2}_{\star}(Y), \Z) = (E^{\Z/2}_{\star}(X))^{\vee},$$
and the cokernel of the first map is $$\mathrm{Ext}^1_{\Z}(\underline{\pi}^{\Z/2}_{\star-1}(Y), \Z) = \mathrm{Ext}^1_{\Z}(E^{\Z/2}_{\star-1}(X), \Z).$$

This shows exactness after applying the evaluation functor $(-)_{\Z/2}$. Exactness for the sequence appearing after applying the functor $(-)_e$ is analogous, taking $Y=X$ in the previous argument.

It suffices to show that the maps appearing in these two short exact sequences are compatible with the restriction and transfer maps of the three Mackey functors.
Recall that restriction and transfer maps come from morphism of $\Z/2$-spectra $\Z/2_+ \rightarrow S^0$ and $S^0 \rightarrow \Z/2_+$, so the compatibility is just an instance of the naturality of the constructions.
\end{proof}

The following observation has powerful consequences.

\begin{pro} \label{pro:relation:dualities}
The diagrams
$$\xymatrix{ {\freem}^{op} \ar[r]^{\nabla} \ar[d]_{H(-)}  & \freem \ar[d]^{H(-)} \\
\Z/2\sh^{op} \ar[r]_{\nabla} & \Z/2\sh}$$
and
$$\xymatrix{ {\torsm}^{op} \ar[r]^{\nabla_{tors}} \ar[d]_{H(-)}  & \torsm \ar[d]^{H(-)} \\
\Z/2\sh^{op} \ar[r]_{ \Sigma^{-1} \nabla} & \Z/2\sh}$$
commute up to a natural isomorphism.
\end{pro}

\begin{proof}
It is essentially a consequence of Proposition \ref{pro:univcoeff}. Consider the case when $M \in Fr_{\Z}\M$, then the restriction to integer gradings of the exact sequence provided by Proposition \ref{pro:univcoeff} for $X = S^0$ and $E = HM$ reduces to an isomorphism of $\Z$-graded Mackey functors
$\underline{\pi}_*(\nabla HM) \cong \nabla \underline{\pi}_{-*}(HM) = \nabla M$ concentrated in degree $0$, since $M$ is torsion-free finitely generated, and thus free. We conclude that $\nabla HM$ is an Eilenberg-MacLane spectrum for the Mackey functor $\nabla M$.

For the second diagram, observe that if $M$ is a finitely generated torsion Mackey functor, then for all $o \in \Or$, $(\nabla M)(o) = \mathrm{Hom}_{\Z}(M(o),\Z) = 0$, so that $\nabla M=0$. The end of the proof is analogous, using the universal property of Eilenberg-MacLane $\Z/2$-spectra.
\end{proof}

\begin{rk}
For the second point, the finite generation hypothesis is superfluous.
\end{rk}

\subsection{Operations and duality}

We now study the relationship between duality and module structures over ring $\Z/2$-spectra. 
\begin{rk}
Results in this direction are also used, in the non-equivariant setting, in the proof of the principal result of \cite{HS14}.
\end{rk}

\begin{pro} \label{pro:dualitymod}
Let $R$ be a ring $\Z/2$-spectrum and $X$ be an $R$-module. 
\begin{enumerate}
\item The $\Z/2$-spectrum $\nabla X$ is naturally an $R$-module.
\item Suppose that $X^e_{\star}$ is a free $\Z$-module, then the natural $R^e_{\star}$-module structure on $(\nabla X)^e_{\star}$ is dual to the one on $X^e_{\star}$.
\item Suppose that $R_{\star} \in Tors_{\Z}\M$, then then the natural $R_{\star}$-module structure on $(\nabla X)_{\star}$ is dual to the one on $X_{\star}$.
\end{enumerate}
\end{pro}

\begin{proof}
\begin{enumerate}
\item Define the $R$-module structure on $\nabla X$ to be the adjoint of the map $\lambda : \nabla X \rightarrow \nabla( R \wedge X) \cong F(R, \nabla X)$. The verification that this map defines a $R$-module structure is routine.
\item Under the hypothesis of freeness on $X^e_{\star}$, the action of an element $r \in R_d^e$ on $(\nabla X)_{\star}^e$ and $(X_{\star}^e)^{\vee}$ fits into a diagram whose rows are isomorphisms
$$ \xymatrix{ (\nabla X)_{\star}^e \ar[r]^{\cong} \ar[d]_r& (X_{\star}^e)^{\vee} \ar[d]^{r} \\
(\nabla X)_{\star+d}^e \ar[r]_{\cong} & (X_{\star+d})^{\vee} } $$
this diagram is commutative because both actions are adjoint to the map $(\nabla X)_{\star}^e  \rightarrow Hom_{\Z}(R_{\star}^e, (\nabla X)_{\star}^e)$ induced by $\lambda$ in homotopy, under the tensor-hom adjunction.
\item The discussion of the previous point can be repeated for the $R_{\star}^e$-module structure on $(\nabla X)_{\star}^e$, which is the dual of the one on $X_{\star}^e$, and the $R_{\star}^{\Z/2}$-module structure on $(\nabla X)_{\star}^{\Z/2}$, which is the dual of the one on $X_{\star}^{\Z/2}$. The fact that the associated maps of Mackey functors $R_{\star} \boxtimes (\nabla X)_{\star} \rightarrow (\nabla X)_{\star}$ are equal comes from the fact that such a map is precisely determined by a pair of actions of $R_{\star}^e$ and $R_{\star}^{\Z/2}$ on $(\nabla X)_{\star}^e$ and $(\nabla X)_{\star}^{\Z/2}$ compatible with the restriction and transfer homomorphisms (see the discussion of Mackey functor maps $M \boxtimes N \rightarrow P$ in \cite[pp.11-12]{FL04}).
\end{enumerate}
\end{proof}

\subsection{First applications and computations}

The key point of the application described in this section relies on the easy $\Z/2$-equivariant cohomology computation.

\begin{lemma}
The Mackey functor $H\mz^{*}(S^{\alpha})$, obtained as the restriction of the $RO(\Z/2)$-graded Mackey functor $H\mz^{\star}(S^{\alpha})$ to integer grading, is concentrated in degree $1$, where it takes the value $[\Z]$.
\end{lemma}

\begin{proof}
Recall that, for all $\Z/2$-space $X$ which admits a cellular decomposition, there is an isomorphism $H\mz_{\Z/2}^*(X) \cong H\Z^*(\frac{X}{\Z/2})$ (see, \cite[Example 3.9]{HHR}), and an isomorphism $H\mz_{e}^*(X) \cong H\Z^*(X)$. Now observe that the underlying space of $S^{\alpha}$ is $S^1$, and the quotient $\frac{S^{\alpha}}{\Z/2}$ is an interval ( and in particular, it is contractible). Thus, the reduced cohomology $H\mz^{*}(S^{\alpha})$ is $[\Z]$ concentrated in degree $1$, and there are isomorphisms $H\mz^{1}_{e}(S^{\alpha}) \cong \Z$ and $H\mz^{1}_{\Z/2}(S^{\alpha}) \cong 0$. There is only one Mackey functor taking these values, which is $[\Z]$.  
\end{proof}

As a consequence, the $\Z/2$-spectrum $\Sigma^{1-\alpha}H\mz$ is an Eilenberg-MacLane $\Z/2$-spectrum for the Mackey functor $[\Z]$: there is a weak equivalence $$\Sigma^{1-\alpha}H\mz \cong H[\Z].$$

\begin{cor} \label{cor:mackselfdual}
There are weak equivalences of $\Z/2$-spectra 
$$H\mz \cong \Sigma^{-2+2\alpha} \nabla H\mz$$ 
and 
$$H\mf \cong \Sigma^{-1+2\alpha} \nabla H\mf.$$
\end{cor}

\begin{proof}
We start by showing the results for the spectrum $H\mz$.
Now, $[\Z]$ is a $\Z$-free Mackey functor satisfying $ \nabla [\Z] \cong [\Z]$, thus, by Proposition \ref{pro:relation:dualities},  $\nabla H[\Z] = H[\Z]$.
We get a chain of weak equivalences
\begin{eqnarray*}
H\mz & \cong & \Sigma^{\alpha-1} H[\Z] \\
& \cong & \Sigma^{\alpha-1} \nabla H[\Z] \\
& \cong & \Sigma^{\alpha-1} \nabla \Sigma^{1-\alpha}H\mz \\
& \cong & \Sigma^{2\alpha-2} \nabla H\mz
\end{eqnarray*}

The cofibre sequence $$H\mz \stackrel{2}{\rightarrow} H\mz \rightarrow H\mf$$ provided by Proposition \ref{pro:functorH} gives $H\mf^{*}(S^{\alpha})= [\F]$, which is a self-dual $\Z$-torsion Mackey functor. The result follows as before, using this time the second diagram of Proposition \ref{pro:relation:dualities}.
\end{proof}

\begin{cor} \label{cor:mackselfdual2}
There is a short exact sequence of Mackey functors 
$$0 \rightarrow \nabla_{tors} (\tors H\mz_{\star-1}) \rightarrow H\mz_{-\star-2+2 \alpha} \rightarrow \nabla( \free H\mz_{\star}) \rightarrow 0.$$
and an isomorphism of Mackey functors $$H\mf_{\star} \cong H\mf_{-\star+2-2\alpha}.$$
\end{cor}

\begin{proof}

For the first point, the short exact sequence of Mackey functors of Proposition \ref{pro:univcoeff} gives an exact sequence
$$0 \rightarrow \nabla_{tors} (\tors H[\Z]_{\star-1}) \rightarrow (\nabla H[\Z])_{-\star}(X) \rightarrow \nabla( \free H[\Z]_{\star}) \rightarrow 0,$$
and the result follows.

The results about $H\mf$ are analogous. 
\end{proof}

\begin{rk}
\begin{itemize}
\item This is an explanation of the symmetry of the computations of coefficients rings made by Ferland and Lewis in \cite[figure 9.1, 9.2, 9.3, 9.4]{FL04}. 
\item We could also have computed the cohomology of $S^{\alpha}$ with the cofibre sequence $\Z/2_+ \rightarrow S^0 \rightarrow S^{\alpha}$ as in \cite{FL04}.
\end{itemize}
\end{rk}

\begin{ex}
As an application, we compute $H\mz_{\star}$. Let $Pos_{\star}$ denotes the sub-Mackey functor of $H\mz_{\star}$ consisting in elements of degree of the form $\Z + \mathbb{N} \alpha$. To understand $H\mz_{\star}$, it is sufficient to compute $Pos_{\star}$ by duality.

First, we know how to compute integral $H\mz$-cohomology of $\Z/2$-spaces, because $H\mz_{\Z/2}^*(X) = H\Z^*(\frac{X}{\Z/2})$ and  $H\mz_{e}^*(X) = H\Z^*(X)$.

In particular, there is an isomorphism of groups, $Pos_{\star}^{e} \cong \Z[\sigma]$, with $\sigma \in Pos_{-1+\alpha}$, and we know the action of $\Z/2$ on it: $\Z/2$ acts by $-1$ on odd powers of $\sigma$. We now compute $Pos_{\star}^{\Z/2}$. Let $k \geq 0$, and consider the cofiber sequence of $\Z/2$-spaces
$$S(k \alpha)_+ \rightarrow S^0 \rightarrow S^{k \alpha}.$$
As an abelian group,
$$H\Z^*(\frac{S(k\alpha)_+}{\Z/2}) =H\Z^*(\R P^{k-1}_+) = \left\{ \begin{matrix} \Z[x]/(2x) & \text{if $k$ even} \\ \Z[x]/(2x) \oplus \Sigma^k \Z & \text{if $k$ odd} \end{matrix} \right.$$
and the map induced by
$S(k \alpha)_+ \rightarrow S^0 $ is the unique injective map in cohomology. We have determined $H\mz^*(S^{k \alpha}) \cong H\Z^{*+1}(\R P^{k-1})$.
Now, $Pos_{\star}$ satisfies Lemma \ref{lemma_mackey_a}, which allows us to recover both the $\Z[a]/(2a)$-module structure on $Pos_{\star}^{\Z/2}$ and the Mackey functor structure of $Pos_{\star}$ (see the "positive part" of figure \ref{fig:coeffz} for a graphical representation of the result).

One concludes the computation of the Mackey functor $H\mz_{\star}$ using Corollary \ref{cor:mackselfdual2}, and $H\mf_{\star}$ using either an analogous proof, or the cofibre exact sequence $$H\mz \stackrel{2}{\rightarrow} H\mz \rightarrow H\mf.$$
A graphical representation of these Mackey functors are represented in figure \ref{fig:coeffz} and figure \ref{fig:coefff} respectively, with the conventions:

\begin{nota} \label{nota:figures}
For simplicity, we will use the following shorthand for some Mackey functors appearing in Definition \ref{de:partmack}.
The symbol $ \bullet$ stands for the Mackey functor $<\F>$, $L$ stands for $L(\F)$, and $L_-$ the Mackey functor $[\F]$. The $\Z[\Z/2]$-module $\Z_-$ is the free abelian group of rank one with the sign action.
A vertical line represents the product with the Euler class $a$. This product induces one of the following Mackey functor maps:
\begin{itemize}
\item the identity of $ \bullet$, 
\item the unique non-trivial morphism $ L \rightarrow \bullet$,
\item the unique non-trivial morphism $ \bullet \inj \mf$.
\end{itemize}
\end{nota}

\begin{figure}[h] 

\begin{center}
\definecolor{qqqqff}{rgb}{0.33,0.33,0.33}

\begin{tikzpicture}[line cap=round,line join=round,>=triangle 45,x=0.6cm,y=0.6cm]
\draw[->,color=black] (-9,0) -- (9,0);
\foreach \x in {-8,-6,-4,-2,2,4,6,8}
\draw[shift={(\x,0)},color=black] (0pt,2pt) -- (0pt,-2pt) node[below] {\footnotesize $\x$};
\draw[->,color=black] (0,-10) -- (0,10);
\foreach \y in {-8,-6,-4,-2,2,4,6,8}
\draw[shift={(0,\y)},color=black] (2pt,0pt) -- (-2pt,0pt) node[left] {\footnotesize $\y$};
\draw[color=black] (0pt,-10pt) node[right] {\footnotesize $0$};
\clip(-9,-9) rectangle (9,9);
\draw (-59.2,24.92) node[anchor=north west] {$\underline{ \mathbb{F}}_2$};
\draw (-58.23,23.91) node[anchor=north west] {$\underline{ \mathbb{F}}_2$};
\draw (-57.22,22.94) node[anchor=north west] {$\underline{ \mathbb{F}}_2$};
\draw (-56.21,21.93) node[anchor=north west] {$\underline{ \mathbb{F}}_2$};
\draw (-55.24,20.92) node[anchor=north west] {$\underline{ \mathbb{F}}_2$};
\draw (-54.23,19.9) node[anchor=north west] {$\underline{ \mathbb{F}}_2$};
\draw (-53.27,18.94) node[anchor=north west] {$\underline{ \mathbb{F}}_2$};
\draw (-52.25,17.93) node[anchor=north west] {$\underline{ \mathbb{F}}_2$};
\draw (-51.29,16.96) node[anchor=north west] {$\underline{ \mathbb{F}}_2$};
\draw (-50.28,15.95) node[anchor=north west] {$\underline{ \mathbb{F}}_2$};
\draw (-49.26,14.99) node[anchor=north west] {$\underline{ \mathbb{F}}_2$};
\draw (-71.15,36.87) node[anchor=north west] {$\underline{ \mathbb{F}}_2$};
\draw (-70.14,35.86) node[anchor=north west] {$\underline{ \mathbb{F}}_2$};
\draw (-69.17,34.95) node[anchor=north west] {$\underline{ \mathbb{F}}_2$};
\draw (-68.21,33.93) node[anchor=north west] {$\underline{ \mathbb{F}}_2$};
\draw (-67.2,32.92) node[anchor=north west] {$\underline{ \mathbb{F}}_2$};
\draw (-66.19,31.91) node[anchor=north west] {$\underline{ \mathbb{F}}_2$};
\draw (-65.22,30.94) node[anchor=north west] {$\underline{ \mathbb{F}}_2$};
\draw (-64.21,29.93) node[anchor=north west] {$\underline{ \mathbb{F}}_2$};
\draw (-63.25,28.92) node[anchor=north west] {$\underline{ \mathbb{F}}_2$};
\draw (-62.23,27.95) node[anchor=north west] {$\underline{ \mathbb{F}}_2$};
\draw (-61.22,26.89) node[anchor=north west] {$\underline{ \mathbb{F}}_2$};
\draw (-0.28,0.76) node[anchor=north west] {$\mz$};
\draw (0.68,-0.25) node[anchor=north west] {$R(\Z_-)$};
\draw (1.69,-1.21) node[anchor=north west] {$\mz$};
\draw (2.71,-2.22) node[anchor=north west] {$R(\Z_-)$};
\draw (3.67,-3.24) node[anchor=north west] {$\mz$};
\draw (4.68,-4.2) node[anchor=north west] {$R(\Z_-)$};
\draw (5.65,-5.21) node[anchor=north west] {$\mz$};
\draw (6.66,-6.23) node[anchor=north west] {$R(\Z_-)$};
\draw (7.62,-7.19) node[anchor=north west] {$\mz$};
\draw (8.64,-8.15) node[anchor=north west] {$R(\Z_-)$};
\draw (9.65,-9.17) node[anchor=north west] {$\mz$};
\draw (-12.24,12.77) node[anchor=north west] {$\mz^{op}$};
\draw (-11.23,11.76) node[anchor=north west] {$L(\Z_-)$};
\draw (-10.26,10.79) node[anchor=north west] {$\mz^{op}$};
\draw (-9.3,9.78) node[anchor=north west] {$L(\Z_-)$};
\draw (-8.29,8.82) node[anchor=north west] {$\mz^{op}$};
\draw (-7.27,7.8) node[anchor=north west] {$L(\Z_-)$};
\draw (-6.31,6.79) node[anchor=north west] {$\mz^{op}$};
\draw (-5.34,5.78) node[anchor=north west] {$L(\Z_-)$};
\draw (-4.33,4.81) node[anchor=north west] {$\mz^{op}$};
\draw (-3.32,3.8) node[anchor=north west] {$L(\Z_-)$};
\draw (-1.3,1.78) node[anchor=north west] {$R(\Z_-)$};
\draw (0,0) -- (0,-10);
\draw (2,-2) -- (2,-10);
\draw (4,-4) -- (4,-10);
\draw (6,-6) -- (6,-10);
\draw (8,-8) -- (8,-10);
\draw (10,-10) -- (10,-10);
\draw (-3,3) -- (-3,10);
\draw (-5,5) -- (-5,10);
\draw (-7,7) -- (-7,10);
\draw (-9,9) -- (-9,10);
\draw (-11,11) -- (-11,10);
\draw (8.0,0.86) node[anchor=north west] {$1$};
\draw (0.1,8.95) node[anchor=north west] {$\alpha$};
\draw (-2.31,2.79) node[anchor=north west] {$L$};
\begin{scriptsize}
\fill [color=qqqqff] (0,-1) circle (1.5pt);
\fill [color=qqqqff] (0,-2) circle (1.5pt);
\fill [color=qqqqff] (0,-3) circle (1.5pt);
\fill [color=qqqqff] (0,-4) circle (1.5pt);
\fill [color=qqqqff] (0,-5) circle (1.5pt);
\fill [color=qqqqff] (0,-6) circle (1.5pt);
\fill [color=qqqqff] (0,-7) circle (1.5pt);
\fill [color=qqqqff] (0,-8) circle (1.5pt);
\fill [color=qqqqff] (0,-10) circle (1.5pt);
\fill [color=qqqqff] (0,-10) circle (1.5pt);
\fill [color=qqqqff] (0,-11) circle (1.5pt);
\fill [color=qqqqff] (2,-3) circle (1.5pt);
\fill [color=qqqqff] (2,-4) circle (1.5pt);
\fill [color=qqqqff] (2,-5) circle (1.5pt);
\fill [color=qqqqff] (2,-6) circle (1.5pt);
\fill [color=qqqqff] (2,-7) circle (1.5pt);
\fill [color=qqqqff] (2,-8) circle (1.5pt);
\fill [color=qqqqff] (2,-9) circle (1.5pt);
\fill [color=qqqqff] (2,-10) circle (1.5pt);
\fill [color=qqqqff] (2,-11) circle (1.5pt);
\fill [color=qqqqff] (4,-5) circle (1.5pt);
\fill [color=qqqqff] (4,-6) circle (1.5pt);
\fill [color=qqqqff] (4,-7) circle (1.5pt);
\fill [color=qqqqff] (4,-8) circle (1.5pt);
\fill [color=qqqqff] (4,-9) circle (1.5pt);
\fill [color=qqqqff] (4,-10) circle (1.5pt);
\fill [color=qqqqff] (4,-11) circle (1.5pt);
\fill [color=qqqqff] (6,-7) circle (1.5pt);
\fill [color=qqqqff] (6,-8) circle (1.5pt);
\fill [color=qqqqff] (6,-9) circle (1.5pt);
\fill [color=qqqqff] (6,-10) circle (1.5pt);
\fill [color=qqqqff] (6,-11) circle (1.5pt);
\fill [color=qqqqff] (8,-9) circle (1.5pt);
\fill [color=qqqqff] (8,-10) circle (1.5pt);
\fill [color=qqqqff] (8,-11) circle (1.5pt);
\fill [color=qqqqff] (10,-11) circle (1.5pt);
\fill [color=qqqqff] (-3,11) circle (1.5pt);
\fill [color=qqqqff] (-3,10) circle (1.5pt);
\fill [color=qqqqff] (-3,9) circle (1.5pt);
\fill [color=qqqqff] (-3,8) circle (1.5pt);
\fill [color=qqqqff] (-3,7) circle (1.5pt);
\fill [color=qqqqff] (-3,6) circle (1.5pt);
\fill [color=qqqqff] (-3,5) circle (1.5pt);
\fill [color=qqqqff] (-3,4) circle (1.5pt);
\fill [color=qqqqff] (-5,11) circle (1.5pt);
\fill [color=qqqqff] (-5,10) circle (1.5pt);
\fill [color=qqqqff] (-5,9) circle (1.5pt);
\fill [color=qqqqff] (-5,8) circle (1.5pt);
\fill [color=qqqqff] (-5,7) circle (1.5pt);
\fill [color=qqqqff] (-5,6) circle (1.5pt);
\fill [color=qqqqff] (-7,11) circle (1.5pt);
\fill [color=qqqqff] (-7,10) circle (1.5pt);
\fill [color=qqqqff] (-7,9) circle (1.5pt);
\fill [color=qqqqff] (-7,8) circle (1.5pt);
\fill [color=qqqqff] (-9,11) circle (1.5pt);
\fill [color=qqqqff] (-9,10) circle (1.5pt);
\end{scriptsize}
\end{tikzpicture}

\end{center}
\caption{The Mackey functor $ H \mz_{ \star}$. Vertical lines represent the product by the Euler class $a$, and dots represent copies of the Mackey functor $<\F>$.}  \label{fig:coeffz}
\end{figure}
\end{ex}

\begin{figure}
 \begin{center}
\definecolor{cqcqcq}{rgb}{0.75,0.75,0.75}
\definecolor{qqqqff}{rgb}{0.33,0.33,0.33}

%\shorthandoff{:}

\begin{tikzpicture}[line cap=round,line join=round,>=triangle 45,x=0.5cm,y=0.5cm]
\draw[->,color=black] (-10,0) -- (10,0);
\foreach \x in {-10,-8,-6,-4,-2,2,4,6,8}
\draw[shift={(\x,0)},color=black] (0pt,2pt) -- (0pt,-2pt) node[below] {\footnotesize $\x$};
\draw[->,color=black] (0,-10) -- (0,10);
\foreach \y in {-10,-8,-6,-4,-2,2,4,6,8}
\draw[shift={(0,\y)},color=black] (2pt,0pt) -- (-2pt,0pt) node[left] {\footnotesize $\y$};
\draw[color=black] (0pt,-10pt) node[right] {\footnotesize $0$};
\clip(-10,-10) rectangle (10,10);
\draw (-0.31,0.5) node[anchor=north west] {$\underline{ \mathbb{F}}$};
\draw (0.7,-0.51) node[anchor=north west] {$\underline{ \mathbb{F}}$};
\draw (1.7,-1.52) node[anchor=north west] {$\underline{ \mathbb{F}}$};
\draw (2.71,-2.52) node[anchor=north west] {$\underline{ \mathbb{F}}$};
\draw (3.67,-3.51) node[anchor=north west] {$\underline{ \mathbb{F}}$};
\draw (4.65,-4.51) node[anchor=north west] {$\underline{ \mathbb{F}}$};
\draw (5.66,-5.49) node[anchor=north west] {$\underline{ \mathbb{F}}$};
\draw (6.66,-6.5) node[anchor=north west] {$\underline{ \mathbb{F}}$};
\draw (7.65,-7.46) node[anchor=north west] {$\underline{ \mathbb{F}}$};
\draw (8.63,-8.47) node[anchor=north west] {$\underline{ \mathbb{F}}$};
\draw (9.64,-9.45) node[anchor=north west] {$\underline{ \mathbb{F}}$};
\draw (-1.81,1.5) node[anchor=north west] {$L_-$};
\draw (-12.25,12.46) node[anchor=north west] {$L$};
\draw (-11.24,11.45) node[anchor=north west] {$L$};
\draw (-10.28,10.49) node[anchor=north west] {$L$};
\draw (-9.3,9.51) node[anchor=north west] {$L$};
\draw (-8.29,8.5) node[anchor=north west] {$L$};
\draw (-7.29,7.5) node[anchor=north west] {$L$};
\draw (-6.3,6.51) node[anchor=north west] {$L$};
\draw (-5.32,5.51) node[anchor=north west] {$L$};
\draw (-4.31,4.53) node[anchor=north west] {$L$};
\draw (-3.31,3.52) node[anchor=north west] {$L$};
\draw (-2.4,2.49) node[anchor=north west] {$L$};
\draw (0,0) -- (0,-10);
\draw (1,-1) -- (1,-10);
\draw (2,-2) -- (2,-10);
\draw (3,-3) -- (3,-10);
\draw (4,-4) -- (4,-10);
\draw (5,-5) -- (5,-10);
\draw (6,-6) -- (6,-10);
\draw (7,-7) -- (7,-10);
\draw (8,-8) -- (8,-10);
\draw (9,-9) -- (9,-10);
\draw (10,-10) -- (10,-10);
\draw (-2,2) -- (-2,10);
\draw (-3,3) -- (-3,10);
\draw (-4,4) -- (-4,10);
\draw (-5,5) -- (-5,10);
\draw (-6,6) -- (-6,10);
\draw (-7,7) -- (-7,10);
\draw (-8,8) -- (-8,10);
\draw (-9,9) -- (-9,10);
\draw (-10,10) -- (-10,10);
\draw (-11,11) -- (-11,10);
\draw (9.4,0.9) node[anchor=north west] {$1$};
\draw (0.1,10) node[anchor=north west] {$\alpha$};
\begin{scriptsize}
\fill [color=qqqqff] (1,-2) circle (1.5pt);
\fill [color=qqqqff] (1,-3) circle (1.5pt);
\fill [color=qqqqff] (1,-4) circle (1.5pt);
\fill [color=qqqqff] (1,-5) circle (1.5pt);
\fill [color=qqqqff] (1,-6) circle (1.5pt);
\fill [color=qqqqff] (1,-7) circle (1.5pt);
\fill [color=qqqqff] (1,-8) circle (1.5pt);
\fill [color=qqqqff] (1,-9) circle (1.5pt);
\fill [color=qqqqff] (1,-10) circle (1.5pt);
\fill [color=qqqqff] (1,-11) circle (1.5pt);
\fill [color=qqqqff] (0,-1) circle (1.5pt);
\fill [color=qqqqff] (0,-2) circle (1.5pt);
\fill [color=qqqqff] (0,-3) circle (1.5pt);
\fill [color=qqqqff] (0,-4) circle (1.5pt);
\fill [color=qqqqff] (0,-5) circle (1.5pt);
\fill [color=qqqqff] (0,-6) circle (1.5pt);
\fill [color=qqqqff] (0,-7) circle (1.5pt);
\fill [color=qqqqff] (0,-8) circle (1.5pt);
\fill [color=qqqqff] (0,-9) circle (1.5pt);
\fill [color=qqqqff] (0,-10) circle (1.5pt);
\fill [color=qqqqff] (0,-11) circle (1.5pt);
\fill [color=qqqqff] (2,-3) circle (1.5pt);
\fill [color=qqqqff] (2,-4) circle (1.5pt);
\fill [color=qqqqff] (2,-5) circle (1.5pt);
\fill [color=qqqqff] (2,-6) circle (1.5pt);
\fill [color=qqqqff] (2,-7) circle (1.5pt);
\fill [color=qqqqff] (2,-8) circle (1.5pt);
\fill [color=qqqqff] (2,-9) circle (1.5pt);
\fill [color=qqqqff] (2,-10) circle (1.5pt);
\fill [color=qqqqff] (2,-11) circle (1.5pt);
\fill [color=qqqqff] (3,-4) circle (1.5pt);
\fill [color=qqqqff] (3,-5) circle (1.5pt);
\fill [color=qqqqff] (3,-6) circle (1.5pt);
\fill [color=qqqqff] (3,-7) circle (1.5pt);
\fill [color=qqqqff] (3,-8) circle (1.5pt);
\fill [color=qqqqff] (3,-9) circle (1.5pt);
\fill [color=qqqqff] (3,-10) circle (1.5pt);
\fill [color=qqqqff] (3,-11) circle (1.5pt);
\fill [color=qqqqff] (4,-5) circle (1.5pt);
\fill [color=qqqqff] (4,-6) circle (1.5pt);
\fill [color=qqqqff] (4,-7) circle (1.5pt);
\fill [color=qqqqff] (4,-8) circle (1.5pt);
\fill [color=qqqqff] (4,-9) circle (1.5pt);
\fill [color=qqqqff] (4,-10) circle (1.5pt);
\fill [color=qqqqff] (4,-11) circle (1.5pt);
\fill [color=qqqqff] (5,-6) circle (1.5pt);
\fill [color=qqqqff] (5,-7) circle (1.5pt);
\fill [color=qqqqff] (5,-8) circle (1.5pt);
\fill [color=qqqqff] (5,-9) circle (1.5pt);
\fill [color=qqqqff] (5,-10) circle (1.5pt);
\fill [color=qqqqff] (5,-11) circle (1.5pt);
\fill [color=qqqqff] (6,-7) circle (1.5pt);
\fill [color=qqqqff] (6,-8) circle (1.5pt);
\fill [color=qqqqff] (6,-9) circle (1.5pt);
\fill [color=qqqqff] (6,-10) circle (1.5pt);
\fill [color=qqqqff] (6,-11) circle (1.5pt);
\fill [color=qqqqff] (7,-8) circle (1.5pt);
\fill [color=qqqqff] (7,-9) circle (1.5pt);
\fill [color=qqqqff] (7,-10) circle (1.5pt);
\fill [color=qqqqff] (7,-11) circle (1.5pt);
\fill [color=qqqqff] (8,-9) circle (1.5pt);
\fill [color=qqqqff] (8,-10) circle (1.5pt);
\fill [color=qqqqff] (8,-11) circle (1.5pt);
\fill [color=qqqqff] (9,-10) circle (1.5pt);
\fill [color=qqqqff] (9,-11) circle (1.5pt);
\fill [color=qqqqff] (10,-11) circle (1.5pt);
\fill [color=qqqqff] (-2,11) circle (1.5pt);
\fill [color=qqqqff] (-2,10) circle (1.5pt);
\fill [color=qqqqff] (-2,9) circle (1.5pt);
\fill [color=qqqqff] (-2,8) circle (1.5pt);
\fill [color=qqqqff] (-2,7) circle (1.5pt);
\fill [color=qqqqff] (-2,6) circle (1.5pt);
\fill [color=qqqqff] (-2,5) circle (1.5pt);
\fill [color=qqqqff] (-2,4) circle (1.5pt);
\fill [color=qqqqff] (-2,3) circle (1.5pt);
\fill [color=qqqqff] (-3,11) circle (1.5pt);
\fill [color=qqqqff] (-3,10) circle (1.5pt);
\fill [color=qqqqff] (-3,9) circle (1.5pt);
\fill [color=qqqqff] (-3,8) circle (1.5pt);
\fill [color=qqqqff] (-3,7) circle (1.5pt);
\fill [color=qqqqff] (-3,6) circle (1.5pt);
\fill [color=qqqqff] (-3,5) circle (1.5pt);
\fill [color=qqqqff] (-3,4) circle (1.5pt);
\fill [color=qqqqff] (-4,11) circle (1.5pt);
\fill [color=qqqqff] (-4,10) circle (1.5pt);
\fill [color=qqqqff] (-4,9) circle (1.5pt);
\fill [color=qqqqff] (-4,8) circle (1.5pt);
\fill [color=qqqqff] (-4,7) circle (1.5pt);
\fill [color=qqqqff] (-4,6) circle (1.5pt);
\fill [color=qqqqff] (-4,5) circle (1.5pt);
\fill [color=qqqqff] (-5,11) circle (1.5pt);
\fill [color=qqqqff] (-5,10) circle (1.5pt);
\fill [color=qqqqff] (-5,9) circle (1.5pt);
\fill [color=qqqqff] (-5,8) circle (1.5pt);
\fill [color=qqqqff] (-5,7) circle (1.5pt);
\fill [color=qqqqff] (-5,6) circle (1.5pt);
\fill [color=qqqqff] (-6,11) circle (1.5pt);
\fill [color=qqqqff] (-6,10) circle (1.5pt);
\fill [color=qqqqff] (-6,9) circle (1.5pt);
\fill [color=qqqqff] (-6,8) circle (1.5pt);
\fill [color=qqqqff] (-6,7) circle (1.5pt);
\fill [color=qqqqff] (-7,11) circle (1.5pt);
\fill [color=qqqqff] (-7,10) circle (1.5pt);
\fill [color=qqqqff] (-7,9) circle (1.5pt);
\fill [color=qqqqff] (-7,8) circle (1.5pt);
\fill [color=qqqqff] (-8,11) circle (1.5pt);
\fill [color=qqqqff] (-8,10) circle (1.5pt);
\fill [color=qqqqff] (-8,9) circle (1.5pt);
\fill [color=qqqqff] (-9,11) circle (1.5pt);
\fill [color=qqqqff] (-9,10) circle (1.5pt);
\fill [color=qqqqff] (-10,11) circle (1.5pt);
\end{scriptsize}
\end{tikzpicture}

%\shorthandon{:}
\end{center}
\caption{The Mackey functor $ H \mf_{ \star}$. Vertical lines represent the product by the Euler class $a$, and dots represent copies of the Mackey functor $<\F>$.}  \label{fig:coefff}

\end{figure}

\begin{nota} \label{nota:sigma}
We call $\sigma^{-1}$ the non trivial element in degree $(1-\alpha)$, so that $(H\mf)_{\Z/2}$ contains $\F[a, \sigma^{-1}]$ as a subalgebra.
\end{nota}
 
\begin{rk}
The class $ \sigma$ is related to the orientation of the tautological fiber bundle on $B\Z/2$ with respect to the modulo $2$ cohomology, that is why powers of $\sigma$ can be seen as orientation classes. This appears in Hu and Kriz computation of $H \mf_{\star}$ in \cite{HK01}.
\end{rk}

\subsection{The slice filtration and the slice spectral sequence}

We now recall the definition of the slice filtration and the slice spectral sequence from \cite{HHR}. However, since our focus is on the group $\Z/2$, the exposition is easier.

\begin{de}
Let $k \in \Z$.
\begin{itemize}
\item The {\it~regular slice sphere} of dimension $2k$ is the $\Z/2$-spectrum $S^{k(1+\alpha)}$.
\item The {\it~induced slice sphere} of dimension $k$ is the $\Z/2$-spectrum $\Z/2_+ \wedge S^k$.
\item The {\it~shifted slice sphere} of dimension $2k-1$ is the $\Z/2$-spectrum $S^{k(1+\alpha)-1}$.
\end{itemize}
The set of slice spheres is the union of the regular, induced and shifted ones.
\end{de}

The construction of the slice tower is analogous to the construction of the Postnikov tower, replacing the role of the $\Z/2$-equivariant $n$ dimensional spheres $S^n$ and $\Z/2_+\wedge S^n$ by the slices spheres of dimension $n$.

\begin{de}
A $\Z/2$-spectrum $X$ is {\it~slice $n$-null} if for every slice sphere $\hat{S}$ of dimension $\geq n$, the $\Z/2$-space $F(\hat{S}, X)$ is contractible. In this case, we note $X<n$ or $X \geq n-1$. A $\Z/2$-spectrum $Y$ is {\it~slice $n$-connective} if for all slice $n$-null spectrum $X$, the $\Z/2$-space $F(Y,X)$ is contractible. We note $Y \geq n+1$ or $Y > n$.
\end{de}

By definition of slice null and slice connective, a $\Z/2$-spectrum $X$ is slice $n$-null if and only if for all $k \in \Z$, $\Sigma^{k(1+\alpha)}X <n+2k$. Similarly, a $\Z/2$-spectrum $Y$ is slice $n$-connective if and only if for all $k \in \Z$, $\Sigma^{k(1+\alpha)}X > n+2k$.

\begin{lemma}
Let $X$ be a $\Z/2$-spectrum and $n \in \Z$. Then $X<n$ if and only if for all $k \geq [n/2]$, $$\underline{\pi}_{k(1+ \alpha)} = 0 = \underline{\pi}_{k(1+ \alpha)+1},$$ where $[-]$ is the floor.
\end{lemma}

\begin{proof}
This is a reformulation of \cite[Lemma 4.14]{HHR}in the special case of the group $\Z/2$, using the fact that for all $l \in \Z$, there is an isomorphism $[\Z/2_+ \wedge S^{2l},X]^{\Z/2} \cong [S^{l(1+\alpha)},X]^e$ and $[\Z/2_+ \wedge S^{2l+1},X]^{\Z/2} \cong [S^{l(1+\alpha)-1},X]^e$.
\end{proof}

Using Bousfield localisation, we see that for all $n \in \Z$, and all $X \in \Z/2\Sp$, there is a unique cofiber sequence
$$ P_{n+1}X \rightarrow X \rightarrow P^n X,$$
such that $P_{n+1}X \geq n+1$ and $P^nX \leq n$.

When it comes to identify the slice filtration of a spectrum, the following lemma is essential.

\begin{lemma}[\emph{\cite[Lemma 4.16]{HHR}}] \label{lemma:unicityslicetower}
Suppose $X$ is a $\Z/2$-spectrum and that
$$\tilde{P}_{n+1} \rightarrow X \rightarrow \tilde{P}^n $$
is a fibration sequence with the property that $\tilde{P}^n \leq n$ and $\tilde{P}_{n+1}X \geq n+1$. Then the
canonical maps $\tilde{P}_{n+1} \rightarrow P_{n+1} X$ and $P^n X \rightarrow  \tilde{P}^n$ are weak equivalences.
\end{lemma}

Since a slice $n$-connective $ \Z/2$-spectrum $X$ is also slice $(n-1)$-connective, there is a natural morphism $P^nX \rightarrow P^{n-1}X$.

\begin{de}
Let $P^n_nX$ be the homotopy fiber of the map $P^nX \rightarrow P^{n-1}X$.
\end{de}

We are now ready to introduce the slice tower properly.

\begin{de}
Let $X$ be a $\Z/2$-spectrum. The {\it~slice tower} of $X$ is the tower of $\Z/2$-spectra
$$ \xymatrix{ \vdots \ar[d] \\ P^{n+1}X \ar[d] \\ P^{n}X \ar[d] \\ P^{n-1}X \ar[d] \\ \vdots}$$
and the {\it~$n$-slice of $X$} is the spectrum $P^n_nX$.
\end{de}

\begin{de} \label{de:slicespectralsequence}
The {\it~slice spectral sequence} of $X$ is the $\Z \times RO(\Z/2)$-graded spectral sequence of Mackey functors associated to the slice tower. It has $E_2$-page:
$$E_2^{s,V}(X) = \underline{\pi}_{V-s}(P_{dim(V)}^{dim(V)}X) \Rightarrow \underline{\pi}_{V-s}(X).$$
With this indexing, the $r^{th}$ differential has degree $(r,r-1)$.
\end{de}

We end the study of slices by a recognition principle for $(-1)$-slices and $0$-slices.

\begin{pro}[\emph{\cite[Proposition 4.50]{HHR}}] A spectrum $X$ is a $(-1)$-slice if and only if it is of the
form $X = \Sigma^{-1}HM$, with $M$ an arbitrary Mackey functor.

A spectrum $X$ is a $0$-slice if and only if it is of the form $HM$ with $M$ a Mackey
functor whose restriction map is a monomorphism.
\end{pro}

\begin{cor}
The spectra $H\mz$ and $H\mf$ are $0$-slices.
\end{cor}

We end this subsection with the multiplicative properties of the slice spectral sequence.

\begin{de}
A $\Z/2$-spectrum $X$ is {\it~pure} if all its slice sections are of the form $\hat{S} \wedge H\mz$, where $\hat{S}$ is a wedge of regular and induced slice spheres.
\end{de}

\begin{pro}[\emph{\cite[Section 4.7]{HHR}}]  \label{pro:multiplicativeprop}
If $X$ and $Y$ are pure $\Z/2$-spectra, then there is a map of spectral sequences
$$E_r^{s,t}(X) \otimes E_r^{s',t'}(Y) \rightarrow E_r^{s+s',t+t'}(X\wedge Y)$$
representing the pairing
$$ \underline{\pi}_{*}(X) \otimes \underline{\pi}_{*}(Y) \rightarrow \underline{\pi}_{*}(X\wedge Y).$$
\end{pro}

\subsection{Relationship between equivariant and non-equivariant Anderson duality}

In this section, we study the relationship between non-equivariant Anderson duality and our $\Z/2$-equivariant version via the forgetful functor $(-)^e$ and fixed points $(-)^{\Z/2}$. As observed in \cite{HS14}, fixed points commutes with Anderson duality only for strongly complete $\Z/2$-spectra. We will make this statement precise in Proposition \ref{pro:fixeddual}.

We first recall briefly the construction of non-equivariant Anderson duality, since this is classical, and quite simpler than its equivariant counterpart.

\begin{de}
Let $I$ be an injective abelian group and $\partial_I$ be the non-equivariant spectrum representing the cohomology theory
$$ X \mapsto Hom_{\Z}(\pi_{-*}(X), I).$$
Let $\partial_{\Z}$ be the homotopy fiber f the map $\partial_{\Q} \rightarrow \partial_{\Q/\Z}$, and $\nabla^e$ be the non-equivariant Anderson duality functor $F(-, \partial_{\Z})$.
\end{de}

As this subsection will be concerned in passing from the non-equivariant stable homotopy category to the equivariant one, we start by fixing some notations about this. In this subsection, one additional category of spectra appear: the category $\Z/2\Sp^{naive}$ of naive spectra (as opposed to the category $\Z/2\Sp$ of genuine spectra). The objects of this category are simply spectra together with an action of $\Z/2$. It is more of a technical point in this paper, so we rather refer to \cite[Introduction, section 1]{Le95} for this notion.

\begin{nota}
Let $\epsilon : \Z/2 \rightarrow \{e\}$ be a group homomorphism and $\epsilon^* : \Sp \rightarrow \Z/2\Sp^{naive}$ the induced functor. This functor takes a non-equivariant spectrum and produces a naive spectra with trivial action of $\Z/2$.
We also consider the change of universe adjunction
$$ i^* : \Z/2\sh \leftrightarrows \Z/2\sh^{naive} : i_*.$$
\end{nota}

Recall that by \cite[Lemma XVI.1.3]{May96}, the unit and counit of the last adjunction is a non-equivariant equivalence. 

We are now ready to study the relationship between $\Z/2$-equivariant and non-equivariant Anderson duality.

\begin{pro} \label{pro:nonequivduality}
There is a canonical map $  i_* \epsilon^*(\partial_{\Z}) \rightarrow d_{\Z}$, which is a non-equivariant equivalence. Consequently, the functors $ \nabla^e$ and $(\nabla( -))^e$ are canonically isomorphic.
\end{pro}

\begin{proof}
We will play with the various adjunction we first build a map between naive $\Z/2$-spectra
$$ \epsilon^* \partial_{\Z} \rightarrow i^*  d_{\Z}.$$

Let $I$ be an injective abelian group. For all $X \in \Z/2 \T$, there is a map 

\begin{eqnarray*}
[X, \epsilon^* \partial_I]_{naive}^{\Z/2} & \cong & [\frac{X}{\Z/2}, \partial_I] \\
& \cong & \mathrm{Hom}_{\Z}(\pi_{-*}(\frac{X}{\Z/2}), I) \\
& \rightarrow &  \mathrm{Hom}_{\Z}(\underline{\pi}^{\Z/2}_{-*}(X),I) \\
& \cong & [X, d_I]^{\Z/2} \\
& \cong & [X, i^* d_I]^{\Z /2}_{naive}, \\
\end{eqnarray*}

where the first isomorphism comes from the adjunction \cite[Proposition 3.12]{MM}, and the third map is induced from the projection $X \rightarrow \frac{X}{\Z/2}$. The last one is a consequence of \cite[Proposition II.1.4]{LMS}, since $X = \Sigma^{\infty}X \cong i_* \Sigma^{\infty}_{naive} X$ with evident notations.

By Yoneda, this gives a map $\epsilon^*\partial_I \rightarrow i^* d_I$, and the adjunction between change of universes functors gives the desired map $  i_* \epsilon^*(\partial_{\Z}) \rightarrow d_{\Z}$.

As a consequence, using the injective resolution $\Z \rightarrow \Q \rightarrow \Q/\Z$ of $ \Z$, this provides a map $i^* \epsilon^* \partial_{\Z} \rightarrow d_{\Z}$.
Now, for all $Y \in \T$, consider the previous map for $X = \Z/2_+ \wedge Y$. This gives a morphism $[Y, \partial_I] \rightarrow [X,d_I]^{\Z/2} \cong [Y, (d_I)^e]$. The fact that the map $\underline{\pi}^{\Z/2}_{-*}(X) \rightarrow \pi_{-*}(\frac{X}{\Z/2})$ is an isomorphism in this case gives the weak equivalence $  (d_{I})^e \cong \partial_{I}$, by Yoneda, and we get the weak equivalence $  (d_{\Z})^e \cong \partial_{\Z}$.

Finally, the functor represented by these two non-equivariant spectra are respectively $(\nabla( -))^e$ and $ \nabla^e$.
\end{proof}

\begin{pro} \label{pro:fixeddual}
Let $E$ be a $\Z/2$-spectrum such that $ \widetilde{E\Z/2} \wedge E \cong 0$. Then, there is a canonical weak equivalence of non-equivariant spectra
$$ (\nabla E)^{\Z/2} \cong \nabla^e (E^{\Z/2}).$$
\end{pro}

\begin{proof}
By hypothesis, $E^{\Z/2} \cong (E\Z/2_+\wedge E)^{\Z/2}$. Now, Adams isomorphism (see for example \cite[Theorem XVI.5.4]{May96}) yields a weak equivalence $$(E\Z/2_+\wedge E)^{\Z/2} \cong \frac{i^*(E\Z/2_+\wedge E)}{\Z/2}.$$
Applying $ \nabla^e$, this gives
\begin{eqnarray*}
\nabla^e(E^{\Z/2}) & \cong & \nabla^e(\frac{i^*(E\Z/2_+\wedge E)}{\Z/2}) \\
& \cong & F(\frac{i^*(E\Z/2_+\wedge E)}{\Z/2}, \partial_{\Z}) \\
& \cong & F(E\Z/2_+\wedge E, i_* \epsilon^*(\partial_{\Z}))^{\Z/2} \\
& \cong & F(E, F(E\Z/2_+ , i_* \epsilon^*(\partial_{\Z})))^{\Z/2}.
\end{eqnarray*}

The map $  i_*\epsilon^*(\partial_{\Z}) \rightarrow d_{\Z}$ provided by Proposition \ref{pro:nonequivduality} induces a weak equivalence $F(E\Z/2_+ , i_*\epsilon^*(\partial_{\Z})) \cong F(E\Z/2_+ , d_{\Z})$. Thus,
\begin{eqnarray*}
 F(E, F(E\Z/2_+ , i^* \epsilon^* (\partial_{\Z})))^{\Z/2} & \cong & F(E, F(E\Z/2_+ , d_{\Z}))^{\Z/2} \\
 & \cong & F(E\Z/2_+ \wedge E , d_{\Z})^{\Z/2} \\
 & \cong & (\nabla E)^{\Z/2},
\end{eqnarray*}
where the last weak equivalence comes from the weak equivalence $E \cong E\Z/2_+ \wedge E$.
\end{proof}

\section{Integral Morava K-theory with reality}

\subsection{Definitions}
In this subsection, we recall the construction of equivariant analogues of integral Morava K-theories, as explained in \cite[section 3]{HK01}. Recall from \textit{loc cit} that there are generators $r_i$, $i>0$ of $M\R_{*(1+\alpha)}$. These are also, by definition, the generators denoted $r_i$ considered in \cite{HHR} (introduced in Lemma 5.33).
Recall that there is a construction of $M\R$-module spectra obtained by {\em killing} a regular ideal $(r_1, \hdots r_n)$, lifting one of $MU_* = \Z[x_i|i \geq 1]$, where $|x_i|=2i$.
Recall also that $M\R_{*(1+\alpha)} \cong \Z[r_i|i \geq 1]$, where $|r_i|=i(1+\alpha)$ are equivariant lifts of the $x_i$.

\begin{de} \label{de:moravareal}
\begin{itemize}
\item We call connective Morava K-theory with reality the $M\R$-module $\Z/2$-spectra $k\R(n)$ obtained by killing the ideal generated by all $r_i$, for all $i \neq 2^n-1$. By construction, there is an $M\R$-module map
$$ M\R \rightarrow k\R(n).$$ Denote by $v_n = r_{2^n-1}$.
\item Consider the $M\R$-module localisation functor $(-)[v_n^{-1}]$ with respect to $v_n$. We call Morava K-theory with reality the $M\R$-module $\Z/2$-spectrum $K\R(n)$ obtained as $k\R(n)[v_n^{-1}]$. It comes with a natural map $M\R[v_n^{-1}] \rightarrow K\R(n)$.
\end{itemize}
\end{de}

\subsection{The slice tower of $K\R(n)_{\star}$}
Our main tool here is the {\em slice spectral sequence} of \cite{HHR} and the various differentials which have been determined by \cite{HK01}.

We recall the structure of the $\Z/2$-equivariant Steenrod algebra.

\begin{de}
Let $\ste^{\star} = H\mf^{\star}_{\Z/2}H\mf$ the algebra of stable $H\mf$-cohomology operations.
\end{de}

Hu and Kriz  \cite{HK01} computed a presentation of the $\Z/2$-equivariant modulo $2$ dual Steenrod algebra
 $$\ste_{\star} = H\mf_{\star}^{\Z/2}H\mf = H\mf_{\star}^{\Z/2}[\xi_{i+1},\tau_i|i\geq0]/I,$$
 where $I$ is the ideal generated by the relation $\tau_i^2 = a\xi_{i+1}+ (a\tau_0+\sigma^{-1})\tau_{i+1}$, and $\sigma^{-1}$ is the class defined in Notation \ref{nota:sigma}.

Recall that there is a duality relating cohomology operations to homology cooperations.

\begin{de}
Denote by $\mathbb{Q}_n : H\mf  \rightarrow \Sigma^{(2^n-1)(1+\alpha)+1} H\mf $ the $H\mf$-cohomology operation dual  to the element $\tau_n \in H\mf_{(2^n-1)(1+\alpha)+1}H\mf$. It is called the $(n+1)^{st}$ Milnor operation. 
\end{de}

\begin{lemma}
The operation $\mathbb{Q}_n$ satisfies $(\mathbb{Q}_n)_{\star}(\sigma^{-2^n}) = a^{2^{n+1}-1}$.
\end{lemma}

\begin{proof}
As $H\mf$ is a ring spectrum, the pair $(H\mf_{\star}^{\Z/2},H\mf_{\star}^{\Z/2}H\mf)$ has a natural structure of a Hopf algebroid. The action of $H\mf^{\star}_{\Z/2}H\mf$ on $H\mf_{\star}^{\Z/2}$ is determined by the unit of the Hopf algebroid $(H\mf_{\star}^{\Z/2},H\mf_{\star}^{\Z/2}H\mf)$. \\
Now, $\eta_R(a)=a$ and $\eta_R(\sigma^{-1}) = \sigma^{-1} + a \tau_0$  together with the fact that $ \eta_R$ is an algebra map gives the desired formula.
\end{proof}

\begin{lemma} \label{lemma:slicekrn}
The slice tower of $K\R(n)$ is
$$ \hdots \Sigma^{k(2^n-1)(1+\alpha)}k\R(n) \stackrel{v_n}{\rightarrow} \Sigma^{(k-1)(2^n-1)(1+\alpha)}k\R(n)  \rightarrow \hdots.$$
the $2k(2^n-1)^{th}$ slice $P^{2k(2^n-1)}(K\R(n))$ is weakly equivalent $\Sigma^{k(2^n-1)(1+\alpha)}H\mz$, and all other slices are contractible.
The composite $$ H\mz \rightarrow \Sigma^{(2^n-1)(1+\alpha)} k\R(n) \rightarrow \Sigma^{(2^n-1)(1+\alpha)+1} H\mz $$
is an integral lift of the Milnor operation $\mathbb{Q}_n : H\mf  \rightarrow \Sigma^{(2^n-1)(1+\alpha)+1} H\mf $, which satisfies $(\mathbb{Q}_n)_{\star}(\sigma^{-2^n}) = a^{2^{n+1}-1}$.
\end{lemma}

\begin{proof}
We already know that
\begin{itemize}
\item the homotopy $\pi_{*}^e(K\R(n))$ is free abelian, as it is the homotopy groups of the non-equivariant $n^{th}$ integral Morava K-theory, which is obtained by killing a regular sequence in $MU_*$,
\item the classes $v_n^k \in \pi_{k(1+\alpha)(2^n-1)}(K\R(n))$ are equivariant refinements of generators of $\pi_{*}^e(K\R(n))$.
\end{itemize}
Thus, by \cite[Proposition 4.64]{HHR}, the canonical maps $$\Sigma^{k(2^n-1)(1+\alpha)}H\mz \rightarrow P_{2k(2^n-1)}^{2k(2^n-1)}(K\R(n))$$ are weak equivalences.

Finally, \cite{HK01} identifies the first $k$-invariant of the connective cover of $K\R(n)$ with a lift of the operation $\mathbb{Q}_n$.
\end{proof}

\begin{rk}
This implies $\pi_0(K\R(n)) \cong \mz$, and in particular $[S^0, K\R(n)]^{\Z/2} \cong \Z$, for degree reasons.
\end{rk}

The result of Lemma \ref{lemma:slicekrn} is expressed in the following diagram, where the dotted arrows represent degree $+1$ maps, {\it~e.g.} $\delta : H\mz \rightarrow \Sigma^{(2^n-1)(1+ \alpha)+1} k\R(n)$.

$$ \xymatrix{ \Sigma^{(2^n-1)(1+\alpha)}k\R(n) \ar[r] \ar[d]_{v_n} & \Sigma^{(2^n-1)(1+\alpha)}H \mz \\
k\R(n) \ar[r] \ar[d]_{v_n} & H \mz  \ar@{-->}[ul]^{ \delta} \ar@{-->}[u]_{ \mathbb{Q}_n} \\
\Sigma^{-(2^n-1)(1+\alpha)} k\R(n) \ar[r] \ar[d]_{v_n} & \Sigma^{-(2^n-1)(1+\alpha)} H \mz \ar@{-->}[ul]^{ \delta} \ar@{-->}[u]_{  \mathbb{Q}_n} \\
\Sigma^{-2(2^n-1)(1+\alpha)} k\R(n) \ar[r] \ar[d]_{v_n} & \Sigma^{-2(2^n-1)(1+\alpha)} H \mz \ar@{-->}[ul]^{ \delta} \ar@{-->}[u]_{  \mathbb{Q}_n} \\
\Sigma^{-3(2^n-1)(1+\alpha)} k\R(n) \ar[r] \ar[d]_{v_n} & \Sigma^{-3(2^n-1)(1+\alpha)} H \mz \ar@{-->}[ul]^{ \delta} \ar@{-->}[u]_{ \mathbb{Q}_n} \\
 \vdots \ar[d] & \\
 K\R(n) & } $$

\begin{cor}
The $E_2$-page of the slice spectral sequence for $K\R(n)_{\star}^{\Z/2}$ is isomorphic to $H\mz_{\star}^{\Z/2}[v_n^{\pm 1}]$.
\end{cor}

\begin{proof}
This is by definition of the $E_2$-page of the slice spectral sequence, as given in Definition \ref{de:slicespectralsequence}.
\end{proof}

We end this section by  completion results for the spectra $K\R(n)$.

\begin{lemma} \label{lemma:completion}
The map $E\Z/2_+ \rightarrow S^0$ induces weak equivalences of $\Z/2$-spectra $$K\R(n) \cong F(E\Z/2_+, K\R(n)),$$
and
$$K\R(n) \cong E\Z/2_+ \wedge K\R(n).$$
\end{lemma}

\begin{proof}
Consider the element $a^{2^{n+1}-1}v_n \in k\R(n)_{\star}^{\Z/2}$. By construction, $k\R(n)$ is a $M\R$ module. Since both $M\R$ and $k\R(n)$ are pure, we can use the multiplicatives properties of the slice spectral sequence described in Proposition \ref{pro:multiplicativeprop}, so the element $a^{2^{n+1}-1}v_n$ is represented on the $E^2$ page by the element $a^{2^{n+1}-1}v_n \in \Sigma^{(2^n-1)(1+\alpha)}H\mz_{\star}^{\Z/2} \subset E^2$. But this element is hit by the first non-trivial differential since $$d_{2^{n+1}-1}(\sigma^{-2^n}) = \mathbb{Q}_n(\sigma^{-2^n})v_n = a^{2^{n+1}-1}v_n.$$ Thus $v_n$ is $a$-torsion, and thus $\pi_{\star}(K\R(n))^{\Z/2}$ is also all $a$-torsion. This has two consequences:
\begin{itemize}
\item one has $\underset{a}{\mathrm{colim}} \ a^n \pi_{\star}(K\R(n)) =0$, so that $\widetilde{E\Z/2} \wedge K\R(n) =0$, the isotropy separation cofibre sequence gives the second weak equivalence.
\item and the $a$-divisible part $div_a(\pi_{\star}^{\Z/2}(K\R(n)))$  of $\pi_{\star}^{\Z/2}(K\R(n))$ is trivial, and $\pi_{\star}^{\Z/2}(F(\widetilde{E \Z/2}, K\R(n))) = div_a(\pi_{\star}^{\Z/2}(K\R(n)))[a^{\pm 1}] = 0$, so that the spectrum $F(\widetilde{E \Z/2}, K\R(n))$ is contractible.
Now, the cofibre sequence
$$ 0 = F(\widetilde{E \Z/2}, K\R(n)) \rightarrow F(S^0, K\R(n)) = K\R(n) \rightarrow F(E \Z/2_+, K\R(n))$$
provides the first desired weak equivalence.
\end{itemize}
\end{proof}

\section{The spectra $K\R(n)$ and their fixed points are self-dual}

In this section, consider the $\Z/2$-spectrum $\nabla K\R(n)$. Apply the exact functor $\nabla$ to the slice tower of $K\R(n)$ to obtain a tower

$$ \xymatrix{ \nabla (\Sigma^{(2^n-1)(1+\alpha)}k\R(n)) \ar@{-->}[dr]^{  \nabla\delta} & \ar[l]  \nabla( \Sigma^{(2^n-1)(1+\alpha)}H \mz) \ar@{-->}[d]_{ \nabla\mathbb{Q}_n} \\
 \nabla(k\R(n)) \ar[r] \ar[u]_{\nabla v_n} \ar@{-->}[dr]^{  \nabla\delta} & \ar[l]   \nabla(H \mz)   \ar@{-->}[d]_{ \nabla\mathbb{Q}_n} \\
 \nabla(\Sigma^{-(2^n-1)(1+\alpha)} k\R(n)) \ar@{-->}[dr]^{  \nabla\delta} \ar[u]_{ \nabla v_n} &   \ar[l] \nabla(\Sigma^{-(2^n-1)(1+\alpha)} H \mz)  \ar@{-->}[d]_{  \nabla \mathbb{Q}_n} \\
 \nabla(\Sigma^{-2(2^n-1)(1+\alpha)} k\R(n)) \ar@{-->}[dr]^{  \nabla\delta}  \ar[u]_{ \nabla v_n} &  \ar[l]  \nabla(\Sigma^{-2(2^n-1)(1+\alpha)} H \mz)  \ar@{-->}[d]_{ \nabla  \mathbb{Q}_n} \\
 \nabla(\Sigma^{-3(2^n-1)(1+\alpha)} k\R(n))   \ar[u]_{ \nabla v_n} &  \ar[l]  \nabla(\Sigma^{-3(2^n-1)(1+\alpha)} H \mz)  \\
 \vdots  & \vdots \\
  \nabla K\R(n) \ar[u]& } $$

in which triangles
$$ \nabla( \Sigma^{k(2^n-1)(1+ \alpha)}H\mz) \rightarrow \nabla( \Sigma^{k(2^n-1)(1+ \alpha)}k\R(n)) \rightarrow \nabla( \Sigma^{(k+1)(2^n-1)(1+ \alpha)}k\R(n))$$
are cofibre sequences, by exactness of $ \nabla$ (Remark \ref{rk:nablaexact}).

\begin{pro} \label{pro:slicenablakrn}
Let $F(n)$ be the fiber of the map $\Sigma^{-2+2\alpha}\nabla K\R(n) \rightarrow \Sigma^{-2+2\alpha}\nabla k\R(n)$. Then, there is a tower over $ \Sigma^{-2+2\alpha} \nabla K\R(n)$:

$$ \xymatrix{ \Sigma^{(2^n-1)(1+\alpha)}F(n) \ar[r] \ar[d] & \Sigma^{(2^n-1)(1+\alpha)}H \mz \\
F \ar[r] \ar[d] & H \Z  \ar@{-->}[ul]^{ \delta} \ar@{-->}[u]_{ \nabla \mathbb{Q}_n} \\
\Sigma^{-(2^n-1)(1+\alpha)} F(n) \ar[r] \ar[d] & \Sigma^{-(2^n-1)(1+\alpha)} H \mz \ar@{-->}[ul]^{ \delta} \ar@{-->}[u]_{ \nabla \mathbb{Q}_n} \\
\Sigma^{-2(2^n-1)(1+\alpha)} F(n) \ar[r] \ar[d]  & \Sigma^{-2(2^n-1)(1+\alpha)} H \mz \ar@{-->}[ul]^{ \delta} \ar@{-->}[u]_{ \nabla \mathbb{Q}_n} \\
\Sigma^{-3(2^n-1)(1+\alpha)} F(n) \ar[r] \ar[d] & \Sigma^{-3(2^n-1)(1+\alpha)} H \mz \ar@{-->}[ul]^{ \delta} \ar@{-->}[u]_{ \nabla\mathbb{Q}_n} \\
 \vdots \ar[d] & \\
\Sigma^{-2+2\alpha} \nabla K\R(n). & } $$
Moreover, the $E_2$-page of the associated spectral sequence converging to $$\left( \Sigma^{-2+2\alpha} \nabla K\R(n) \right)_{\star}^{\Z/2}$$ is isomorphic to $ H\mz_{\star}^{\Z/2}[v_n^{\pm 1}]$.
\end{pro}

\begin{proof}
The existence of such a diagram is a consequence of the isomorphism $H\mz \cong \Sigma^{-2+2 \alpha} \nabla H\mz$, provided by Corollary \ref{cor:mackselfdual}, and the octahedral  axiom.
\end{proof}

\begin{rk}
This implies $[S^{2-2 \alpha}, \nabla K\R(n)]^{\Z/2} \cong \Z$, for degree reasons.
\end{rk}

\begin{rk} \label{rk:slicenablekrn}
By the unicity property of the slice filtration provided by Lemma \ref{lemma:unicityslicetower}, this tower is actually the slice tower for $\nabla K\R(n)$.
\end{rk}

Recall from Proposition \ref{pro:dualitymod} that $ \nabla K\R(n)$ is naturally a $M\R$-module spectrum.

\begin{thm} \label{thm:selfand}
The map $M\R \rightarrow \Sigma^{-2+2\alpha} \nabla K\R(n)$ induced by a generator $1 \in \pi_0(\Sigma^{-2+2\alpha}K\R(n))$ factorizes through $K\R(n)$, and yields a weak equivalence
$$ K\R(n) \stackrel{\cong}{\rightarrow} \Sigma^{-2+2\alpha} \nabla K\R(n).$$
\end{thm}

\begin{proof}
Choose a map
$$f : S^0 \rightarrow \Sigma^{-2+2\alpha} \nabla K\R(n)$$
representing a generator of $\pi_0(\Sigma^{-2+2\alpha} \nabla K\R(n))$.
By adjunction, $f$ induces a $M\R$-module map
$M\R \rightarrow \Sigma^{-2+2\alpha} \nabla K\R(n)$,
but $K\R(n)$ is a $v_n$-local $M\R$-module, thus the $M\R$ module $\nabla K \R(n)$ is also local by Proposition \ref{pro:dualitymod}, and we get a map
$M\R[v_n^{\pm 1}] \rightarrow \Sigma^{-2+2\alpha} \nabla K\R(n)$.

Now, by construction, the endomorphism of $K\R(n)$ induced by $r_k \in M\R_{k(1+\alpha)}$, for $k \neq (2^n-1)$ is trivial. By Proposition \ref{pro:dualitymod}, the map $$\Sigma^{k(1+ \alpha)}\nabla K\R(n) \rightarrow \nabla K\R(n)$$ is also trivial. Consequently, the map of $M\R$-modules $$M\R \rightarrow \Sigma^{-2+2\alpha}K\R(n)$$ factorizes through $k\R(n)$, and gives a $M\R$-module map
$$ k\R(n) \rightarrow \Sigma^{-2+2\alpha} \nabla K\R(n).$$
Now, by construction, $\Sigma^{-2+2\alpha} \nabla K\R(n)$ is $v_n$-local. This provides the map
$$\phi : K\R(n) \rightarrow \Sigma^{-2+2\alpha} \nabla K\R(n).$$

It remains to show that $\phi$ induces an isomorphism of $RO(\Z/2)$-graded homotopy groups.

We know that $K\R(n)_{*(1+\alpha)}^{\Z/2} \cong \Z[v_n^{\pm 1}]$ and $\Z[v_n^{\pm 1}] \cong \Sigma^{-2+2\alpha}\nabla K\R(n)_{*(1+\alpha)}^{\Z/2}$ by the spectral sequence of Proposition \ref{pro:slicenablakrn} (which is the slice spectral sequence for $\Sigma^{-2+2\alpha} \nabla K\R(n)$ by Remark \ref{rk:slicenablekrn}), and the slice spectral sequence for $K\R(n)$. Moreover, this is an isomorphism of $\Z[v_n]$-modules because this is a map of $M\R$-modules, every $\Z/2$-spectrum in play is pure, so this is a consequence of the multiplicative properties of the slice spectral sequence (Proposition \ref{pro:multiplicativeprop}).

By construction, $\pi_{\star}^{ \Z/2}(\phi)$ is a $\Z[v_n]$-module map, again by Proposition \ref{pro:multiplicativeprop}. As a consequence, $\phi$ induces an isomorphism
$$ \pi_{*(1+\alpha)}^{ \Z/2}(\phi) : K\R(n)_{*(1+\alpha)}^{\Z/2} \cong \left(\Sigma^{-2+2\alpha} \nabla K\R(n)\right)_{*(1+\alpha)}^{\Z/2}.$$

But the $E_2$-pages of both slice spectral sequences are free $H\mz_{\star}$-modules, generated in degrees multiple of $(1+\alpha)$, so $ \phi$ induces an isomorphism of $H \mz_{ \star}$-modules between the $E_2$-pages of the slice spectral sequences for $K \R(n)$ and $\Sigma^{-2+2 \alpha} \nabla K\R(n)$. The spectral sequence morphism induced by $ \phi$ is an isomorphism, thus $\phi$ is a weak equivalence.
\end{proof}

We now turn to the study of the $\Z/2$-equivariant dual of the higher chromatic analogues of $KO$.

\begin{de}
Let $KO(n)$ be the non-equivariant spectrum $(K\R(n))^{\Z/2}$.
\end{de}

\begin{rk} We already know by Lemma \ref{lemma:completion} and Proposition \ref{pro:fixeddual} that 
\begin{eqnarray*}
\nabla^e(KO(n)) & = & \nabla^e\left( (K\R(n))^{\Z/2}\right) \\
&\cong & \left(\nabla K\R(n)\right)^{\Z/2} \\
& \cong & \left(\Sigma^{-2+2\alpha} K\R(n)\right)^{\Z/2}.
\end{eqnarray*}
Unfortunately, there is {\it~a priori} no reason to identify the non-equivariant spectrum $\left(\Sigma^{-2+2\alpha} K\R(n)\right)^{\Z/2}$ with a shift of $KO(n)$. This identification if the subject of the end of this paper.
\end{rk}

\begin{lemma}
The $\Z/2$-spectrum $K\R(n)$ is $(2^{n+1}-2^{n+1}\alpha)$-periodic.
\end{lemma}

\begin{proof}
Consider the element $ \sigma^{-2^{n+1}}v_n \in H \mz_{2^{n+1}-2^{n+1}\alpha + (2^n-1)(1+ \alpha)} \subset E^2(M\R)$ where $E^2(M\R)$ is the $E_2$-page of the slice spectral sequence for $M \R$. Recall that, by \cite[Corollary 9.13]{HHR}, the element $ \sigma^{-2^{n+1}}v_n \in E^2(M\R)$ is a permanent cycle.

Denote $F(E\Z/2_+,E^2(K\R(n)))$ the $E_2$-page of the spectral sequence induced by the tower obtained by the slice tower applying the functor $F(E\Z/2_+,-)$.
Now, the morphism of spectral sequences $E^2(K\R(n)) \rightarrow F(E\Z/2_+,E^2(K\R(n)))$ determines the differentials in $F(E\Z/2_+,E^2(K\R(n)))$ by the same formul\ae as in $E^2(K\R(n))$. This assures the convergence of $F(E\Z/2_+,E^2(K\R(n)))$ to $F(E\Z/2_+,K\R(n))_{\star}$.
The map 
$$ \Sigma^{-2^{n+1}(1- \alpha) + (2^n-1)(1+ \alpha)} F(E\Z/2_+,E^2(K\R(n))) \rightarrow F(E\Z/2_+,E^2(K\R(n)))$$ 
induced by the product with $\sigma^{-2^{n+1}}v_n$ via the $M\R[v_n^{-1}]$-module structure coincides with the multiplication with $ \sigma^{-2^{n+1}}$ on the copies of 
$$(F(E\Z/2_+,H \mz))_{\star}^{\Z/2_+} \cong \Z[a, \sigma^{\pm 2}]/(2a),$$
which is an isomorphism of $\Z$-modules. Thus, the product by $\sigma^{-2^{n+1}}v_n$ induces an isomorphism of the spectral sequence $F(E\Z/2_+,E^2(K\R(n)))$.
We conclude that the product by $\sigma^{-2^{n+1}}$ on $K\R(n)$ gives the desired periodicity isomorphism, by the $v_n$-periodicity of $K\R(n)$.
\end{proof}

\begin{cor} \label{cor:koselfand}
Let $KO(n) = (K\R(n))^{ \Z/2}$. There is a weak equivalence
$$ KO(n) \cong \Sigma^{-2^{n+2}+4} \nabla^e KO(n).$$
\end{cor}

\begin{proof}
Passage to fixed points in the chain of weak equivalences
$$ \nabla K\R(n) \cong \Sigma^{2-2 \alpha} K \R(n) \cong \Sigma^{-2^{n+1}+2 +(2^{n+1}-2)\alpha}K\R(n) \stackrel{v_n^2}{\cong} \Sigma^{4-2^{n+2}}K \R(n)$$
shows that
$$ (\nabla K\R(n))^{\Z/2} \cong (\Sigma^{4-2^{n+2}} K\R(n))^{\Z/2} \cong \Sigma^{4-2^{n+2}} (K\R(n))^{\Z/2}.$$

Now, Lemma \ref{lemma:completion} and Proposition \ref{pro:fixeddual} gives a weak equivalence $$(\nabla K\R(n))^{\Z/2} \cong \nabla^e( (K\R(n))^{\Z/2}) = \nabla^e(KO(n)).$$
\end{proof}

%\bibliography{biblio}

\begin{thebibliography}{1}

\bibitem[And]{And69}
D.W. Anderson.
\newblock Universal coefficient theorems for k-theory, univ. california,
  berkeley, calif. (1969).

\bibitem[Ati66]{At66}
M.~F. Atiyah.
\newblock {$K$}-theory and reality.
\newblock {\em Quart. J. Math. Oxford Ser. (2)}, 17:367--386, 1966.

\bibitem[BC76]{BC76}
Edgar~H. Brown, Jr. and Michael Comenetz.
\newblock Pontrjagin duality for generalized homology and cohomology theories.
\newblock {\em Amer. J. Math.}, 98(1):1--27, 1976.

\bibitem[FL04]{FL04}
Kevin~K. Ferland and L.~Gaunce Lewis, Jr.
\newblock The {$R{\rm O}(G)$}-graded equivariant ordinary homology of
  {$G$}-cell complexes with even-dimensional cells for {$G={\Bbb Z}/p$}.
\newblock {\em Mem. Amer. Math. Soc.}, 167(794):viii+129, 2004.

\bibitem[HHR14]{HHR}
M.~A. {Hill}, M.~J. {Hopkins}, and D.~C. {Ravenel}.
\newblock {On the non-existence of elements of Kervaire invariant one}.
\newblock {\em ArXiv e-prints}, July 2014.

\bibitem[HK01]{HK01}
Po~Hu and Igor Kriz.
\newblock Real-oriented homotopy theory and an analogue of the
  {A}dams-{N}ovikov spectral sequence.
\newblock {\em Topology}, 40(2):317--399, 2001.

\bibitem[HS14]{HS14}
D.~{Heard} and V.~{Stojanoska}.
\newblock {K-theory, reality, and duality}.
\newblock {\em ArXiv e-prints}, January 2014.

\bibitem[Kai71]{Kai71}
Paul~C. Kainen.
\newblock Universal coefficient theorems for generalized homology and stable
  cohomotopy.
\newblock {\em Pacific J. Math.}, 37:397--407, 1971.

\bibitem[Lew95]{Le95}
L.~Gaunce Lewis, Jr.
\newblock Change of universe functors in equivariant stable homotopy theory.
\newblock {\em Fund. Math.}, 148(2):117--158, 1995.

\bibitem[LMSM86]{LMS}
L.~G. Lewis, Jr., J.~P. May, M.~Steinberger, and J.~E. McClure.
\newblock {\em Equivariant stable homotopy theory}, volume 1213 of {\em Lecture
  Notes in Mathematics}.
\newblock Springer-Verlag, Berlin, 1986.
\newblock With contributions by J. E. McClure.

\bibitem[May96]{May96}
J.~P. May.
\newblock {\em Equivariant homotopy and cohomology theory}, volume~91 of {\em
  CBMS Regional Conference Series in Mathematics}.
\newblock Published for the Conference Board of the Mathematical Sciences,
  Washington, DC, 1996.
\newblock With contributions by M. Cole, G. Comeza{\~n}a, S. Costenoble, A. D.
  Elmendorf, J. P. C. Greenlees, L. G. Lewis, Jr., R. J. Piacenza, G.
  Triantafillou, and S. Waner.

\bibitem[MM02]{MM}
M.~A. Mandell and J.~P. May.
\newblock Equivariant orthogonal spectra and {$S$}-modules.
\newblock {\em Mem. Amer. Math. Soc.}, 159(755):x+108, 2002.

\bibitem[{Sto}11]{Sto11}
V.~{Stojanoska}.
\newblock {Duality for Topological Modular Forms}.
\newblock {\em ArXiv e-prints}, May 2011.

\bibitem[TW95]{TW95}
Jacques Th{\'e}venaz and Peter Webb.
\newblock The structure of {M}ackey functors.
\newblock {\em Trans. Amer. Math. Soc.}, 347(6):1865--1961, 1995.

\end{thebibliography}

%\bibliographystyle{alpha}

\end{document}